\documentclass[12pt,reqno]{amsart}

\usepackage{tikz}
\usetikzlibrary{positioning, arrows.meta}
\usepackage{epsfig}
\usepackage{amscd}
\usepackage[mathscr]{eucal}
\usepackage{amssymb}
\usepackage{amsxtra}
\usepackage{amsmath}
\usepackage[all]{xy}
\usepackage{mathtools}
\usepackage[pdfencoding=auto]{hyperref}
\usepackage{bookmark}
\usepackage{hyperref}
\usepackage{caption}
\usepackage{colortbl}
\usepackage{arydshln}
\DeclarePairedDelimiter\abs{\lvert}{\rvert}

\theoremstyle{plain}
\newtheorem{mainthm}{Theorem}

\newtheorem{prop}{Proposition}[section]
\newtheorem{cor}[prop]{Corollary}

\usepackage{amsthm}
\newtheorem{theorem}{Theorem}[section]
\newtheorem{lemma}[theorem]{Lemma}
\newtheorem{conj}[theorem]{Conjecture}

\theoremstyle{definition}
\newtheorem{dfn}[subsection]{Definition}
\newtheorem*{dfnnonum}{Definition}

\theoremstyle{remark}
\newtheorem{rem}[subsection]{Remark}

\begin{document}

\bigskip

\title[A new discriminant for the $Z$-function and the corrected Gram's law]{A new discriminant for the $Z$-function and the corrected Gram's law}
\date{\today}

\author{Yochay Jerby}

\address{Yochay Jerby, Faculty of Sciences, Holon Institute of Technology, Holon, 5810201, Israel}
\email{yochayj@hit.ac.il}


%
%
% ...
\begin{abstract} 
Let $Z(t)=\zeta \left ( \frac{1}{2} + it \right ) e^{-i \theta(t)}$ be the Hardy $Z$-function, where $\zeta(\sigma +it)$ is the Riemann zeta function and $\theta(t)$ is the Riemann-Siegel $\theta$-function. In 1903 Gram introduced Gram's law $$(-1)^n Z(g_n)>0,$$ where $g_n$ is the unique solution of $\theta(g_n) = \pi n$. If the law would have been satisfied by all Gram points it would have implied the Riemann hypothesis. However, it is known to be violated for infinite Gram points $g_n$, called bad Gram points. 

In this work we introduce, for any $n \in \mathbb{Z}$, a discriminant function $\Delta_n ( \overline{a})$ in the variables $\overline{a} \in \mathbb{R}^{N(n)}$ where $N(n)=\left [ \frac{g_n}{2} \right ]$. We prove that the Riemann hypothesis (RH) holds if and only if the following corrected Gram's law holds $$(-1)^n \Delta_n (\overline{1}) >0.$$ While our discriminant $\Delta_n(\overline{a})$ is a highly non-linear function for a given $n \in \mathbb{Z}$, we prove that its second-order approximation is given by  
$$
\Delta_n(\overline{1}) \approx Z(g_n) +(-1)^n \left (  \frac{Z'(g_n)}{\ln \left ( \frac{g_n }{2 \pi } \right )}  \right )^2. 
$$ In particular, the RH can be viewed as encoded in the higher-order terms of $\Delta_n (\overline{1})$ for bad Gram points $g_n$. This leads  to the numerical discovery of the following new dynamic repulsion property: If $g_n$ is an isolated bad Gram point then $\abs{Z'(g_n)}\geq 4\abs{Z(g_n)}$. We conjecture that this new dynamic repulsion relation holds for any $n \in \mathbb{Z}$ and discuss its strong relations to the Montgomery pair correlation conjecture.   
\end{abstract}

\maketitle
%
%
\section{Introduction}
\label{s:1}

\subsection{The $A$-philosophy for the $Z$-function} The Riemann zeta function is given by $\zeta(s) = \sum_{n=1}^{\infty} n^{-s}$ in the region $\operatorname{Re}(s) > 1$. In his seminal 1859 manuscript \cite{R}, Riemann extended $\zeta(s)$ to a meromorphic function of $s \in \mathbb{C}$ on the whole complex plane, with a single pole at $s=1$. The Hardy $Z$-function, denoted as \(Z(t)\), is the real function
\begin{equation}
Z(t) = e^{i \theta(t)} \zeta \left( \frac{1}{2} + it \right)
\end{equation}
where \(\theta(t)\) is the Riemann-Siegel \(\theta\)-function, given by the equation
\begin{equation}
\theta(t) = \arg \left( \Gamma \left( \frac{1}{4} + \frac{it}{2} \right) \right) - \frac{t}{2} \log t,
\end{equation}
see \cite{E, I}. The zeros of \(Z(t)\) are a central object of study, and questions regarding them represent some of the most profound open questions in mathematics. The three questions of interest in this work are schematically presented in Fig. \ref{fig:f1}:
\begin{figure}[ht!]
\centering
\includegraphics[scale=0.3]{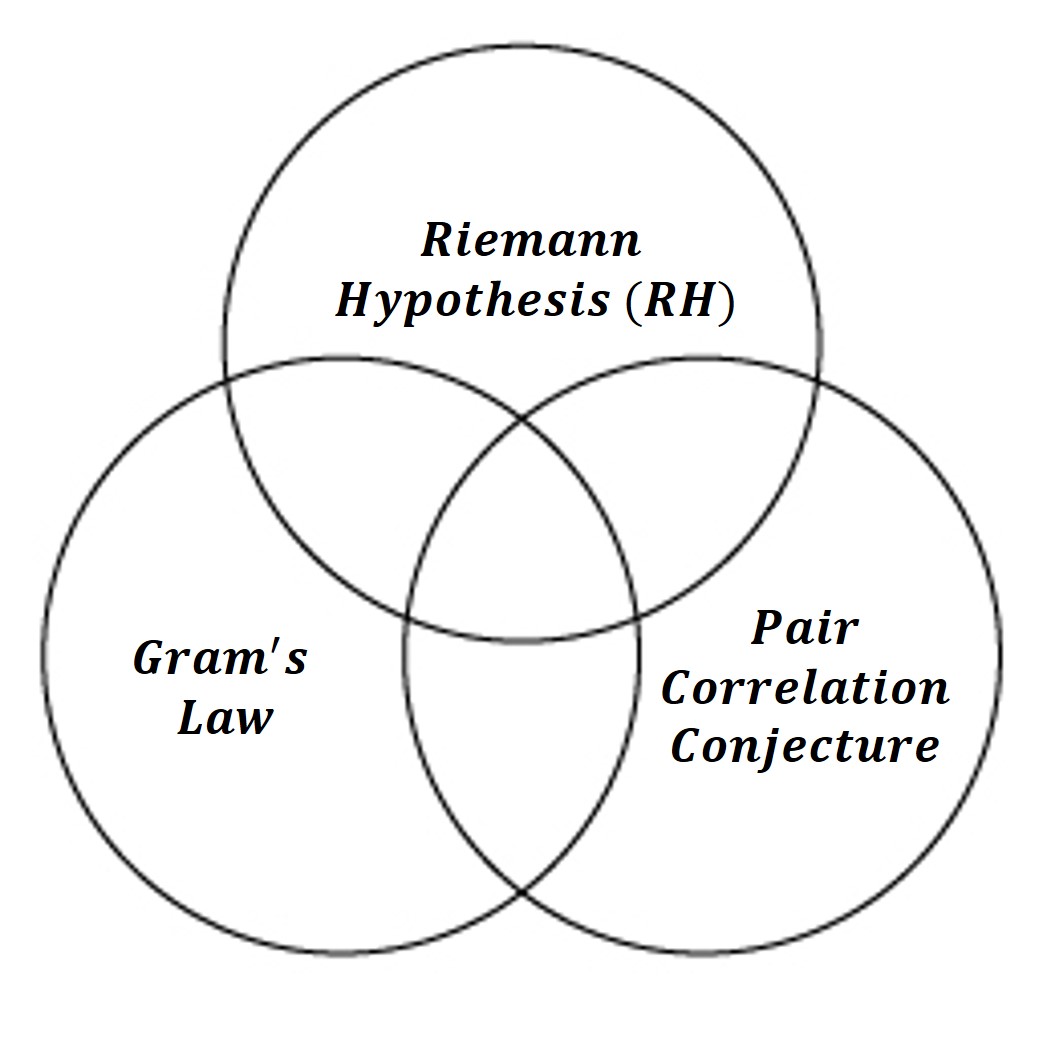}
\caption{\small Schematic presentation of questions regarding zeros of $Z(t)$: Riemann hypothesis, Gram's law and Montgomery's pair correlation conjecture.}
\label{fig:f1}
\end{figure}

Each of the three questions highlights central features of the zeros of \(Z(t)\) that, although extensively validated numerically, are still considered lacking a comprehensive unified theoretical framework for their study:

\begin{description}
 
\item[(I) The Riemann Hypothesis] The Riemann Hypothesis suggests that all non-trivial zeros of \(\zeta(s)\) lie on the critical line \(\operatorname{Re}(s) = \frac{1}{2}\), equivalent to the non-trivial zeros of \(Z(t)\) being real. Resolving this hypothesis could revolutionize our understanding of prime number distribution and has implications across mathematics and physics, including number theory, algebra, cryptography, and quantum physics \cite{Borwein2008Riemann, RHP}. The conjecture has been extensively numerically verified \cite{Gourdon2004, Le, Odlyzko1992, Od, OdlyzkoSchoenhage1988, PT,Ro, Titchmarsh1935, Turing1953}. However, any theoretical plausibility argument for it is currently considered lacking.

\item[(II) Gram's Law] For any integer \(n \in \mathbb{Z}\), the \(n\)-th Gram point \(g_n\) is defined as the unique solution to 
\begin{equation}
\label{eq:Gram}
\theta(g_n) = \pi n.
\end{equation}
In \cite{G} Gram introduced the observation known as Gram's law, which asserts that generally
\begin{equation}
\label{eq:Gram-law}
(-1)^n Z(g_n) > 0.
\end{equation}
If \eqref{eq:Gram-law} would have hold for any $n \in \mathbb{Z}$, it would have implied the RH. The first violation of the law at $n=126$ was discovered in 1925 by Hutchinson in \cite{Hu}. While the law has been studied by many authors, see for instance \cite{BB, Kor2, Kor1, Shan, Tru1, Tru2}, understanding the complete role and implications of Gram's law, particularly the theoretical distinctions between Gram points that satisfy the law and those that do not, remains an elusive question.

\item[(III) The Montgomery Pair Correlation Conjecture] Introduced in 1973, the conjecture suggests that the non-trivial zeros of $Z(t)$ exhibit spacings analogous to the eigenvalue spacings of large random Hermitian matrices from the Gaussian Unitary Ensemble (GUE). Precisely, for \(\alpha \leq \beta\) define
\begin{equation}
A(T ; \alpha, \beta) := \left \{ (\rho,\rho') \mid 0 < \rho, \rho' < T \text{ and } \frac{2 \pi \alpha}{\ln(T)} \leq \rho - \rho' \leq \frac{2 \pi \beta}{\ln(T)} \right \}.
\end{equation} The conjecture states that under the assumption of RH the following holds
\begin{equation}
\label{eq:pcint}
N(T ; \alpha, \beta) := \sum_A 1 \sim \left( \int_{\alpha}^{\beta} \left( 1 - \frac{\sin(\pi u)}{\pi u} \right) \, du + \delta_0([\alpha, \beta]) \right) \frac{T}{2 \pi} \ln \left  (\frac{T}{2 \pi e}  \right )
\end{equation}
as \(T \to \infty\), see \cite{M}. Since its introduction, it has played a critical role in linking non-linear phenomena from number theory, random matrix theory, and quantum chaos. Over the years, numerous positive results aligning the properties of zeta zeros with predictions from random matrix theory have been confirmed, and vice versa \cite{Berry,BeKe,Bog2,Bog1,BogKeat1,BogKeat2,KS,Od,RS}. Additionally, Andrew Odlyzko's seminal empirical studies in 1987 provided robust numerical support for the conjecture \cite{Od}. Nonetheless, the underlying reasons for the connection between \(Z(t)\) and random matrices, as predicted by the conjecture, continue to be a mystery.
\end{description}

The \( A \)-philosophy, as introduced by Gelfand, Kapranov, and Zelevinsky in \cite{GKZ}, advocates for studying mathematical objects not in isolation but rather in relation to a broader parameter space. This approach becomes particularly powerful when analyzing the discriminant hypersurface that forms within the parameter space, often revealing essential insights about the original mathematical object, its zeros and their realness. 

From a geometric point of view, discriminants are typically defined as points in parameter space where both the function and its derivatives vanish, which is equivalent, in the case of discrete zeros, to the hypersurface of elements where a multiple zero occurs. The challenge in analyzing discriminants lies in the fact that, while they offer an intuitive geometric definition, their concrete computation often presents a non-trivial task, and even in the algebraic case, one encounters various open questions.

In \cite{J5} we have introduced the parameter space $\mathcal{Z}_N$ \emph{of $A$-variations of $Z(t)$}, whose elements are functions of the form 
\begin{equation}
\label{eq:Z-sections} 
Z_N(t ; \overline{a} ) = \cos(\theta(t))+ \sum_{k=1}^{N} \frac{a_k}{\sqrt{k+1} } \cos ( \theta (t) - \ln(k+1) t),
\end{equation}
where $\overline{a} = (a_1,...,a_N) \in \mathbb{R}^N$. The elements of $\mathcal{Z}_N$ are variations of $Z(t)$ in the sense that one has the following approximation 
\begin{equation} 
Z(t) = Z_{N(t)}(t ; \overline{1})+ O \left ( \frac{1}{\sqrt[4]{t}} \right ),  
\end{equation} where $N(t) = \left [ \frac{t}{2} \right ]$ and hence the elements $Z_N(t;\overline{a})$ for other $\overline{a} \in \mathbb{R}^N$ can be considered as variations of $Z(t)$ in the corresponding range $2N \leq t \leq 2N+2$. 

Our aim in this work is to apply the $A$-philosophy to the transcendental setting of the Hardy $Z$-function and the parameter space $\mathcal{Z}_N$. Remarkably, this leads to the description of a unified theoretical framework for the study of zeros of $Z(t)$, shedding new light on the classical questions described above and their interconnections, as well as to the discovery of various new fundamental properties of zeros of $Z(t)$.

\subsection{The $\bf n$-th Gram Discriminant and the RH} A multiple zero occurs in instances where an extremal point of a function is also a zero of the function. In \cite{J5} we have introduced the core function of $Z(t)$ to be the function \begin{equation}  
  Z_0(t):=cos(\theta(t)),
  \end{equation} at the origin of $\mathcal{Z}_N$. The extremal points of $Z_0(t)$ are the solutions of the equation 
\begin{equation} 
Z'_0(t) = \theta'(t) \cdot sin (\theta(t)) =0, 
\end{equation}   
  and hence are given by the Gram points, $g_n$, for any $n \in \mathbb{Z}$, as defined by \eqref{eq:Gram}. Locally, we can extend $g_n(\overline{a})$ to be the extremal point of $Z_N(t; \overline{a})$ corresponding to $g_n$ around the origin. This leads us to define the main object of study in this work:
\begin{dfnnonum}[The \(n\)-th Gram discriminant]
For any \(n \in \mathbb{Z}\), we define 
\begin{equation} 
\Delta_n(\overline{a}) := Z_{N(g_n)}(g_n(\overline{a}); \overline{a})
\end{equation}
as the \emph{\(n\)-th Gram discriminant of \(Z(t)\)}.
\end{dfnnonum}

   Our first results is:  
\begin{mainthm}[Corrected Gram's law equivalent to RH] \label{mainthm:B} For any $n \in \mathbb{Z}$, The Riemann hypothesis holds if and only if the following corrected Gram's law holds 
\begin{equation} \label{eq:corrected0}
(-1)^n \Delta_n (\overline{1} ) >0.
\end{equation}  In particular, the extended Gram point $g_n(\overline{a})$ can be analytically continued to $\overline{1}=(1,...,1)$.  
\end{mainthm}

Theorem \ref{mainthm:B} demonstrates that both the RH and Gram's law are intimately related to the geometric properties of the discriminant \(\Delta_n(\overline{a})\). Hence, it is instrumental to develop efficient methods for evaluating \(\Delta_n(\overline{a})\). However, it should be noted that both \(\Delta_n(\overline{a})\) and the extended Gram point \(g_n(\overline{a})\) are transcendental and a-priori highly non-linear functions of \(\overline{a} \in \mathbb{R}^{N(g_n)}\), making the attainment of a fully closed-form expressions for them practically infeasible.

 Nevertheless, one of the advantages of the parameter space $\mathcal{Z}_N$ is that it allows for the computation of derivatives. We prove:

\begin{mainthm}[First-order approximation of $\Delta_n(\overline{a})$] \label{thm:B} For any $n \in \mathbb{Z}$ and $\overline{a}\in \mathbb{R}^{N(g_n)}$ the following first-order approximation holds 
\begin{equation}
 \Delta_n(\overline{a}) =Z(g_n;\overline{a})+ O(\abs{\overline{a}}^2). 
 \end{equation} 
\end{mainthm}

As mentioned above, one of the enduring mysteries surrounding Gram's law since its introduction is that although it proves useful for validating the RH in many cases, it fails at various Gram points and thus lacks the sensitivity required to definitively establish the RH in general. Theorem \ref{mainthm:B} suggests that the classical law discovered by Gram is actually an approximation of the corrected law \eqref{eq:corrected0}. In particular, Theorem \ref{thm:B} shows that the observed tendency of the classical law to hold can now be explained by the fact that in many cases, the linear first-order approximation afforded by \(Z(g_n) \approx Z(g_n ; \overline{1}) \) serves as a sufficient approximation of our discriminant \(\Delta_n(\overline{1})\).

\subsection{Second-order approximation of $\Delta_n(\overline{a})$ and Repulsion} Theorem \ref{mainthm:B} and \ref{thm:B} show that Gram's  classical law is essentially the linear first-order approximation of our corrected Gram's law. Thus, we are especially interested in the contents of the corrected Gram's law \eqref{eq:corrected0} for those bad Gram points for which the classical law fails. Let us present the information encoded in the non-linear second-order approximation 
\begin{equation}
\Delta_n(r) = Z(g_n ; \overline{r} ) +\frac{1}{2} H_n \cdot r^2+O(r^3),
\end{equation}
for the special case of the $1$-parametric path $\overline{r}=(r,...,r)$ in $\mathcal{Z}_N$, where 
\begin{equation}
H_n:=\sum_{k_1,k_2=1}^N \frac{\partial^2 \Delta_{n}}{\partial a_{k_1} \partial a_{k_2}} (\overline{0}) 
\end{equation} 
is the Hessian of second derivatives. The general case of $\overline{a}=(a_1,...,a_N)$ is proved in the paper. We have:

\begin{mainthm}[Second-order Hessian of $\Delta_n(r)$] \label{thm:B2} For any $n \in \mathbb{Z}$ the second-order Hessian of $\Delta_n(r)$ is given by  
\begin{equation} 
H_n =2(-1)^n \left (  \frac{Z'(g_n)}{\ln \left ( \frac{g_n }{2 \pi } \right )}  \right )^2. 
\end{equation}  
\end{mainthm} 

Theorem \ref{thm:B2} shows that the second-order Hessian $H_n$ is proportional to $(Z'(g_n ))^2$. Consider the gradient of $g_n(\overline{a})$ at $ \overline{a}=\overline{0}$ given by  
\begin{equation}
\nabla  g_{n} (\overline{0}) := \left (  \frac{\partial g_{n}}{\partial a_1} (\overline{0}) ,...,\frac{\partial g_{n}}{\partial a_N} (\overline{0}) \right ),
\end{equation} 
which measures the direction of steepest decent of $g_n(\overline{a})$ in the space $\mathcal{Z}_N$. The following result shows the variational interpretation of $Z'(g_n )$ in terms of the gradient:
\begin{mainthm} \label{thm:C} For any $n \in \mathbb{Z}$ the following holds
\begin{equation}
Z'(g_n) = \frac{1}{4} (-1)^n \ln^2 \left ( \frac{g_n}{2 \pi} \right ) \overline{1} \cdot \nabla g_n(0). 
\end{equation}
In particular, the second-order Hessian $H_n$ measures the magnitude of the local shift of $g_n(r)$ along the $t$-axis. 
\end{mainthm} 

An essential feature of the Montgomery pair correlation conjecture is that, due to the decay of the integral \eqref{eq:pcint} for small \(u\), it is often interpreted as suggesting the existence of a statistical repulsion phenomenon between consecutive zeros of \(Z(t)\). Such a property would indicate that zeros are less clustered than in a purely random distribution, mirroring behaviors observed in non-linear systems of repelling particles. However, similar to the conjecture itself, a theoretical explanation for this expected repulsion phenomenon is still considered to be lacking. It is also important to emphasize that although the term repulsion suggests an inherent dynamic behaviour, in the original setting of the pair-correlation the zeros of $Z(t)$ are considered static and fixed, not moving entities interacting over time. Rather, this anticipated repulsion phenomena is understood as a statistical trend observable across the entire set of zeros when examined as a whole and not on the level of individual zeros.

Theorem \ref{thm:B2} and \ref{thm:C} combined reveal that in order for the corrected Gram's law \eqref{eq:corrected0} to hold for bad Gram points $g_n$, the non-linear second-order Hessian $H_n$ is expected to have the role of a correcting term requiring a strong move of the position of $g_n(r)$ compensating for the violation of the classical law. In particular, we identify the following two aforementioned forces: The pull towards the axis expressed by $Z(g_n)$, measured by the linear first-order term of $\Delta_n(r)$, and the shift along the axis expressed by $Z'(g_n)$, measured by $H_n $. This expected correlation between the values of $Z(g_n)$ and $Z'(g_n)$ for bad Gram points leads us to discover the following new experimental repulsion phenomena which, to the best of our knowledge, is the first property observed to be uniquely satisfied by bad Gram points: 
\begin{conj}[Repulsion for isolated bad Gram points] \label{conj:rep} Assume $g_n$ is a bad Gram point with good consecutive neighbours $g_{n \pm 1}$. Then the following (non-sharp) bound holds
\begin{equation} \label{eq:rep}
\abs{Z'(g_n) }>4 \abs{Z(g_n)}.
\end{equation} 
\end{conj}

The violation of the classical Gram's law for bad Gram points, $g_n$, indicates according to Theorem \ref{mainthm:B} and \ref{thm:B} a tendency of the discriminant \(\Delta_n (\overline{a})\) to change signs at the first-order level which is, in turn, interpreted as a tendency of the two consecutive zeros, \(t_n( \overline{a})\) and \(t_{n+1}( \overline{a})\), to collide. We therefore refer to \eqref{eq:rep} as a repulsion relation since the correlation it expresses between \(Z(g_n)\) and \(Z'(g_n)\) suggests that when the consecutive zeros corresponding to $g_n$ tend to collide, there must also an infinitesimal balancing force acting to push them apart as well. In fact, we argue that this newly discovered dynamic repulsion phenomena \eqref{eq:rep} should be considered as the underlying property responsible for the statistical repulsion anticipated by the Montgomery pair-correlation conjecture as well as even the RH itself. 

Our analysis of the \(A\)-variation space \(\mathcal{Z}_N\) introduces a wealth of new results on the zeros of \(Z(t)\), establishing a novel unified theoretical framework for their study via the properties of the associated discriminants \(\Delta_n(\overline{a})\). We not only cast new light on classical questions such as Gram's law, the pair-correlation conjecture, and the RH but also uncover previously unknown properties, including the repulsion relation \eqref{eq:rep}. The discussion and proofs of Theorems \ref{mainthm:B}, \ref{thm:B},\ref{thm:B2} to \ref{thm:C} are provided throughout this paper, highlighting the need and potential for developing new methods to further explore the rich new properties of variations of zeros of $Z(t)$.

\bigskip

The rest of the work is organized as follows: In Section \ref{s:2}, we recall the calculation of $Z(t)$ via approximate functional equations, the \(A\)-philosophy and explain the definition of the variation space. Section \ref{s:3} introduces the local discriminant \(\Delta_n(\overline{a})\) of two consecutive zeros and proves the corrected Gram's law of Theorem \ref{mainthm:B}. In Section \ref{s:5}, Theorem \ref{thm:B}, describing the first-order approximation of \(\Delta_n(\overline{a})\), is proved. Section \ref{s:6} computes \(H_n\), establishing the second-order approximation of \(\Delta_n(\overline{a})\) and the proof of Theorem \ref{thm:B2}. In Section \ref{s:6.5}, the relation between the second-order Hessian and shifts of $g_n(\overline{a})$ along the $t$-axis is studied and Theorem \ref{thm:C} is proven. Section \ref{s:7} describes the discovery of the new repulsion phenomena for isolated Gram points and discusses its relation to the Montgomery pair-correlation. In Section \ref{s:7.5} we revisit Edwards' speculation and describe the application of the repulsion relation to RH. Section \ref{s:8} presents a summary and concluding remarks.

\section{Review of the $A$-Variation Space and Edwards' Speculation} 
\label{s:2} 

In this section, we review relevant results on approximate functional equations, the \(A\)-variation space, and Edwards' speculation. This review is mainly based on \cite{J4,J5} and is included for completeness.

\subsection{The Hardy-Littlewood AFE and the Core Function} In the series of works \cite{HL,HL2,HL3} Hardy and Littlewood developed the formula known as the approximate functional equation (AFE). In its most widely used form the formula is given by   
\begin{equation} 
\label{eq:HL} 
Z(t) =  2 \sum_{k=0}^{\widetilde{N}(t)-1} \frac{cos(\theta(t)-ln(k+1) t) }{\sqrt{k+1}}+R(t),
\end{equation}  
where $\widetilde{N}(t) = \left [ \sqrt{ \frac{ t}{2 \pi}} \right ]$ and the error term is given by 
\begin{equation} 
R(t)=O \left ( \frac{1}{\sqrt[4]{t}}  \right ).
\end{equation}
This representation, while powerful, for many purposes, requires further refinement of the error term $R(t)$ to enhance the level of precision, a subject not addressed in the original works of Hardy and Littlewood. The zero-th order term of the AFE is of special importance: 

\begin{dfn}[Core Function]
We refer to the zero-th order term of the AFE 
\begin{equation} Z_0(t) := cos (\theta (t)) 
\end{equation}
as the core function of $Z(t)$.
\end{dfn} 

Figure \ref{fig:f0} presents a comparison between $\ln \left | Z(t) \right | $ (in blue) and the core $\ln \left | Z_0(t) \right | $ (in orange) in the range $0 \leq t \leq 50$:

\begin{figure}[ht!]
\centering
\includegraphics[scale=0.5]{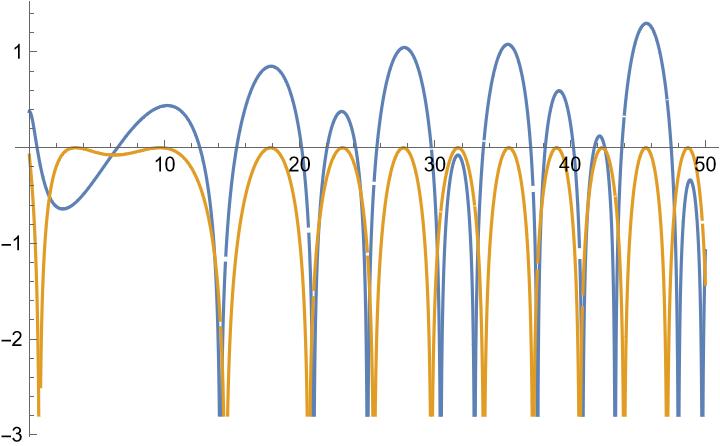}
\caption{\small{$\ln \left | Z(t) \right |$ (blue) and $\ln \left | Z_0 (t) ) \right | $ (orange) for $0 \leq t \leq 50$.}}
\label{fig:f0}
\end{figure}

The fact that the zeros of the core $Z_0(t)$ can be considered as rough approximations of the zeros of $Z(t)$ was actually observed by various authors, see for instance \cite{E,FL,J,SP1}, and might have already been known to Riemann himself. It is known that for any \( n \in \mathbb{Z} \), the zeros of the \( Z_0(t) \) are given by 
 \begin{equation}
 \label{eq:tn0}
    t_n = \frac{(8 n - 11) \pi}{4 \cdot W_0 \left ( \frac{8 n - 11}{8 \cdot e}\right )} \in \mathbb{R}, 
\end{equation}
where $W_0(x)$ is the principal branch of the Lambert function \cite{FL}. In particular, the core $Z_0(t)$ satisfies RH in the sense that all of its non-trivial zeros are real. Consequently, Edwards suggested in \cite{E} that the study of the RH can be reframed as a question regarding the extent of deviation of \(Z(t)\) from its core \(Z_0(t)\), a perspective we refer to as 'Edwards' Speculation'. However, within the context of the classical AFE itself, transforming 'Edwards' Speculation' into a formally described approach for the RH faces various substantial challenges, as detailed at length in \cite{J5}.

\subsection{The $A$-parameter Space and Edwards' Speculation} In \cite{SP2,SP1} Spira introduced the notion of sections of the AFE
\begin{equation}
\label{eq:Section} 
Z_N(t) :=Z_0(t)+\sum_{k=1}^{N-1} \frac{\cos(\theta(t)-\ln(k+1) t) }{\sqrt{k+1}},
\end{equation}
and studied their properties for different values of $N \in \mathbb{N}$. He noted that aside from the AFE \eqref{eq:HL}, which can be written as 
\begin{equation}
Z(t)=2Z_{\widetilde{N}(t)}(t) + O \left ( \frac{1}{\sqrt[4]{t}} \right ),
\end{equation}
 sections also give the following additional high-order approximation
\begin{equation}
\label{eq:Spira}
Z(t)=Z_{N(t)}(t) + O \left ( \frac{1}{\sqrt[4]{t}} \right ),
\end{equation}
with considerably more terms $N(t) = \left [\frac{t}{2} \right ]$, to which we refer as \emph{the Spira approximation}. 

In \cite{J4} we have explored the remarkable properties of this Spira approximation \eqref{eq:Spira}, and showed that contrary to the classical \eqref{eq:HL}, the Spira approximation is sufficiently sensitive to discern the RH. In particular, based on our previous work \cite{J}, we have proven:

\begin{theorem}[\cite{J4}] \label{thm:Spira-RH} All the non-trivial zeros of the Hardy $Z$-function $Z(t)$ are real if and only if all the non-trivial zeros of the high-order Spira approximation $Z_{N(t)}(t)$ are real. 
\end{theorem}

In \cite{J5} we introduced: 

\begin{dfn}[\( N \)-th $A$-Parameter Space]
For a given \( N \in \mathbb{N} \), we define \( \mathcal{Z}_N \) to be the $A$-parameter space consisting of functions of the form
\begin{equation}
\label{eq:var} 
Z_N(t; \overline{a}) = Z_0(t) + \sum_{k=1}^{N-1} \frac{a_k}{\sqrt{k+1}} \cos (\theta(t) - \ln(k+1) t),
\end{equation}
where \( \overline{a} = (a_1, \ldots, a_N) \) belongs to \( \mathbb{R}^N \). 
\end{dfn}

Within the $A$-parameter space $\mathcal{Z}_N$ we thus have two unique elements of special importance. The first one being the 
the core $Z_0(t) = Z_N (t ; \overline{0})$, which lies at the origin $\overline{a}=\overline{0}$ of $\mathcal{Z}_N$. The second being  
 $Z_N(t; \overline{1})$ at $\overline{a}=\overline{1}$, which according to the the high-order Spira approximation satisfies $Z(t) \approx Z_N(t ; \overline{1})$ in $2N \leq t \leq 2N+2$.

Let $\gamma(r)$ be  a smooth parametrized curve in $\mathcal{Z}_N$ for $r \in [0, 1]$, connecting the core function $Z_0(t)$, at $r = 0$, to $Z_N (t; \overline{1})$, at $r=1$. Locally, for small enough $r>0$, every zero $t_n$ of $Z_0(t)$ can be smoothly extended to a unique zero $t_n(r)$ of $Z_N(t ; \gamma(r))$. In general, one can always extend the definition of $t_n(r)$ in a continuous, yet not necessarily smooth and unique manner.

\begin{theorem}[Edwards' Speculation for High-Order Sections \cite{J5}] \label{thm:ESHO}
The Riemann Hypothesis holds if and only if, for any \(n \in \mathbb{Z}\), there exists a path \(\gamma_n(r)\) in \(\mathcal{Z}_{N(t_n) }\) from \(\overline{a}=\overline{0}\) to \(\overline{a}=\overline{1}\), along which \(t_n\) does not collide with its adjacent zeros \(t_{n\pm 1}\). That is, for which \(t_n(r) \neq t_{n \pm 1}(r)\) for all \(r \in [0,1]\).
\end{theorem}

A priori, one might expect that the zeros \(t_n(r)\) can move arbitrarily in the complex plane, around \(t_n\). However, Theorem \ref{thm:ESHO} demonstrates that the zeros of elements in \(\mathcal{Z}_N\) actually exhibit behaviour similar to that of polynomials with real coefficients. The RH is also seen to be intimately related to the set of parameters at which collisions of the form $t_n (r) = t_{n \pm 1}(r)$ occur, leading us to introduce the $n$-th discriminant.

\section{The $n$-th Discriminant $\Delta_n (\overline{a})$ and the Corrected Gram's Law} 
\label{s:3}

The classical discriminant of a quadratic polynomial \(F(z ; \overline{a}) = a_2 z^2 + a_1 z + a_0\) with real coefficients $\overline{a} = ( a_0,a_1,a_2)$ is given by \(\Delta(\overline{a}) = a_1^2 - 4a_0 a_2\), which is itself a quadratic expression in the parameters \(\overline{a}\). Note that the extremal point of \(F(z ; \overline{a})\) is \( g(\overline{a}) = -\frac{a_1}{2a_2}\) and hence 
\begin{equation}
F(g(\overline{a}) ; \overline{a}) = \frac{a_1^2}{4a_2} - \frac{a_1^2}{2a_2} + a_0 = -\frac{a_1^2}{4a_2} + a_0 = 0 \Leftrightarrow \Delta(\overline{a}) = a_1^2 - 4a_0 a_2 = 0.
\end{equation}

The discriminant $\Delta(\overline{a})$ thus measures how far $F(t ; \overline{a})$ is far from having a multiple zero, which occurs in instances where an extremal point of a function is also a zero of the function. 

In order to generalize to $\mathcal{Z}_N$, note that the extremal points of the core $Z_0(t)$ are the solutions of the equation 
\begin{equation} 
Z'_0(t) = \theta'(t) \cdot sin (\theta(t)) =0, 
\end{equation}   
  and hence are given by the Gram points, $g_n$, for any $n \in \mathbb{Z}$, as defined by \eqref{eq:Gram}. Locally, we can extend $g_n(\overline{a})$ to be the extremal point of $Z_N(t; \overline{a})$ corresponding to $g_n$ around the origin. This leads us to define the main object of study in this work:
\begin{dfnnonum}[The \(n\)-th Gram discriminant]
For any \(n \in \mathbb{Z}\), we define 
\begin{equation} 
\Delta_n(\overline{a}) := Z_{N(g_n)}(g_n(\overline{a}); \overline{a})
\end{equation}
as the \emph{\(n\)-th Gram discriminant of \(Z(t)\)}.
\end{dfnnonum}

For instance, Fig. \ref{fig:f2} illustrates the graphs of the $1$-parametric family
\begin{equation}
Z_1(t; r) := \cos(\theta(t)) + \frac{r}{\sqrt{2}} \cos( \theta(t) - \ln(2) t) \in \mathcal{Z}_1,
\end{equation}
for \(t\) in the range \(66.5 \leq t \leq 70\) and \(r=0\) (blue), \(r=0.75\) (orange), \(r=1.5\) (green), and \(r=2.25\) (red), showing  how the zeros of $Z_1(t;r)$ evolve similarly to polynomials with real coefficients:

\begin{figure}[ht!]
	\centering
		\includegraphics[scale=0.4]{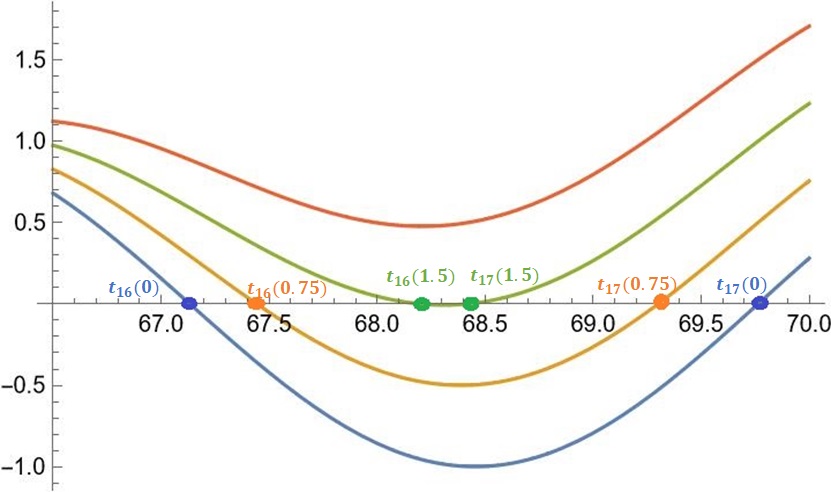}
	\caption{\small{$\ln \left |Z_1(t; a) \right |$ in the range $66.5 \leq t \leq 70$ for $r=0$ (blue), $r=0.75$ (orange), $r=1.5$ (green) and $r=2.25$ (red).}}
\label{fig:f2}
	\end{figure} 
	
Figure \ref{fig:f2} shows the collision process between the $16$-th and $17$-th zeros of $Z_1(t;r)$ starting at the zeros $t_{16}(0)=t_{16}$ and $t_{17}(0)=t_{17}$ of $Z_0(t)$ for $r=0$. As $r$ grows, the two zeros $t_{16}(r)$ and $t_{17}(r)$ approach each other until they collide to a double zero around $r \approx 1.6$ after which they would leave the real line together, as two conjugate complex zeros, which are of course no longer presented in the real graph. We have:  
\begin{theorem}[Corrected Gram's law equivalent to RH] \label{mainthm:B2} For any $n \in \mathbb{Z}$, The Riemann hypothesis holds if and only if the following corrected Gram's law holds 
\begin{equation} \label{eq:corrected}
(-1)^n \Delta_n (\overline{1} ) >0.
\end{equation}  In particular, the extended Gram point $g_n(\overline{a})$ can be analytically continued to $\overline{1}=(1,...,1)$.  
\end{theorem}

\begin{proof}
According to Theorem \ref{thm:ESHO} the RH holds if and only if, for any \(n \in \mathbb{Z}\), there exists a path \(\gamma_n(r)\) in \(\mathcal{Z}_{N(t_n) }\) from \(\overline{a}=\overline{0}\) to \(\overline{a}=\overline{1}\), along which \(t_n\) does not collide with its adjacent zeros \(t_{n\pm 1}\). In particular, along such a curve the discriminant $\Delta_n(\gamma_n(r))$ does not change sign. 
Note that 
\begin{equation}
\Delta_n(\overline{0}) = Z_N(g_n; 0) = \cos(\theta(g_n))=cos(\pi n) = (-1)^{n}.
\end{equation} 
Hence, the fact that $\Delta_n(\gamma_n(r))$ does not change sign implies \eqref{eq:corrected} as required. 
\end{proof}

 Note that, contrary to the classical Gram law \eqref{eq:Gram-law}, which is an empirical observation regarding the numerical tendency of Gram points, the corrected Gram's law \eqref{eq:corrected} is expected to hold for all Gram points and is equivalent to the RH.
 
 While our discriminants \(\Delta_n(\overline{a})\) share much similarity with discriminants of polynomials with real coefficients, they are nevertheless transcendental functions of the space \(\mathcal{Z}_{N(n)}\), and hence a closed-form expression for them is not expected.
 At this point enters the second fundamental feature of the $A$-philosophy, aside from the existence of discriminants, which is the ability to study smooth variations through derivatives. In the next sections we will describe results regarding the geometrical content of the first and second derivatives.

 \section{The First-Order Approximation of the Corrected Law is the Classical Law} 
 \label{s:5}

Let us consider the first-order approximation of $\Delta_n (\overline{a})$ given by 
\begin{equation}
\Delta_n(\overline{a}) = \Delta_{n}(\overline{0})+ \nabla \Delta_{n}  (\overline{0}) \cdot \overline{a}+O(\abs{\overline{a}}^2),
\end{equation} 
where the gradient vector is given by 
\begin{equation}
\nabla  \Delta_{n} (\overline{0}) := \left (  \frac{\partial \Delta_{n}}{\partial a_1} (\overline{0}) ,...,\frac{\partial \Delta_{n}}{\partial a_N} (\overline{0}) \right ).
\end{equation} 
We have: 
 
\begin{theorem}[First order approximation is Gram's law] \label{thm:first-order} For any $n \in \mathbb{Z}$, the first-order approximation of the discriminant $\Delta_n(\overline{a})$ of the linear curve is given by  
\begin{equation} 
\Delta_{n}(\overline{a}) = Z(g_n ; \overline{a} ) +O(\abs{\overline{a}}^2).
\end{equation} 
In particular, the classical Gram's law is the first-order approximation of the corrected Gram's law for the linear curve. 
\end{theorem} 

\begin{proof} 
Consider the function 
\begin{equation}
F_k(t;a):= \cos(\theta(t)) + \frac{a}{\sqrt{k+1}} \cos(\theta(t)- \ln(k+1)t)
\end{equation} 
and set 
\begin{equation} 
G_k(t;a):= \frac{\partial}{\partial t} F_k(t; a). 
\end{equation}
Denote by $g_n(a)$ the extremal point of $F_k(t;a)$ locally extending the gram point $g_n$. Then, by definition, the discriminant can be written as 
\begin{equation} 
\Delta_n( 0,...,0,a,0,...,0) = F_k(g_n(a);a). 
\end{equation}
Hence, by the chain rule, we have 
\begin{equation} 
\frac{\partial \Delta_{n}}{\partial a_k} (\overline{0})=\frac{\partial}{\partial a} F_k(g_n(a) ; a)(0)= \frac{\partial g_n}{\partial a}(0) \cdot G_k(g_n ; 0)+ \frac{\partial}{\partial a} F_k (g_n ; 0). 
\end{equation}
But since the Gram points are exactly the solutions of $G_k(g_n;0)=0$, we get  
\begin{equation} 
\frac{\partial \Delta_{n}}{\partial a_k} (\overline{0})=\frac{\partial}{\partial a} F_k (g_n ; 0)=\frac{1}{\sqrt{k+1} } \cos ( \theta (g_n) - \ln(k+1) g_n).
\end{equation}
The first-order approximation of the discriminant $\Delta_n(r)$ is given by   
\begin{multline} 
\Delta_{n}(\overline{a}) \approx \Delta_{n}(\overline{0})+ \nabla \Delta_{n}  (\overline{0})  \cdot \overline{a}= \\ = 
\Delta_{n}(\overline{0})+ \cdot \sum_{k=1}^N a_k \cdot \frac{\partial \Delta_{n}}{\partial a_k} (\overline{0})= \\ = \cos(\theta(g_n)) +\sum_{k=0}^N \frac{a_k}{\sqrt{k+1} } \cos ( \theta (g_n) - \ln(k+1) g_n) = Z(g_n ; r).
\end{multline} 
\end{proof}

Theorem \ref{thm:first-order} implies that the classical Gram's law \eqref{eq:Gram-law} can be seen as the tendency of the discriminant, $\Delta_n(1)$, and its first-order approximation, $Z(g_n)$, to be of similar sign, for many Gram points $g_n$. Recall that a Gram point \(g_n\) is said to be \emph{good} if it satisfies Gram's law, \( (-1)^n Z(g_n) > 0 \). Otherwise, it is considered \emph{bad}. Figure \ref{fig:f3} shows $\Delta_n (r)$ (blue) and its first order approximation $Z_N (g_n ; r)$ (orange) for the good Gram point $n=90$ (left) and the bad Gram point $n=126$ (right), with $0 \leq r \leq 1$:

\begin{figure}[ht!]
	\centering
		\includegraphics[scale=0.35]{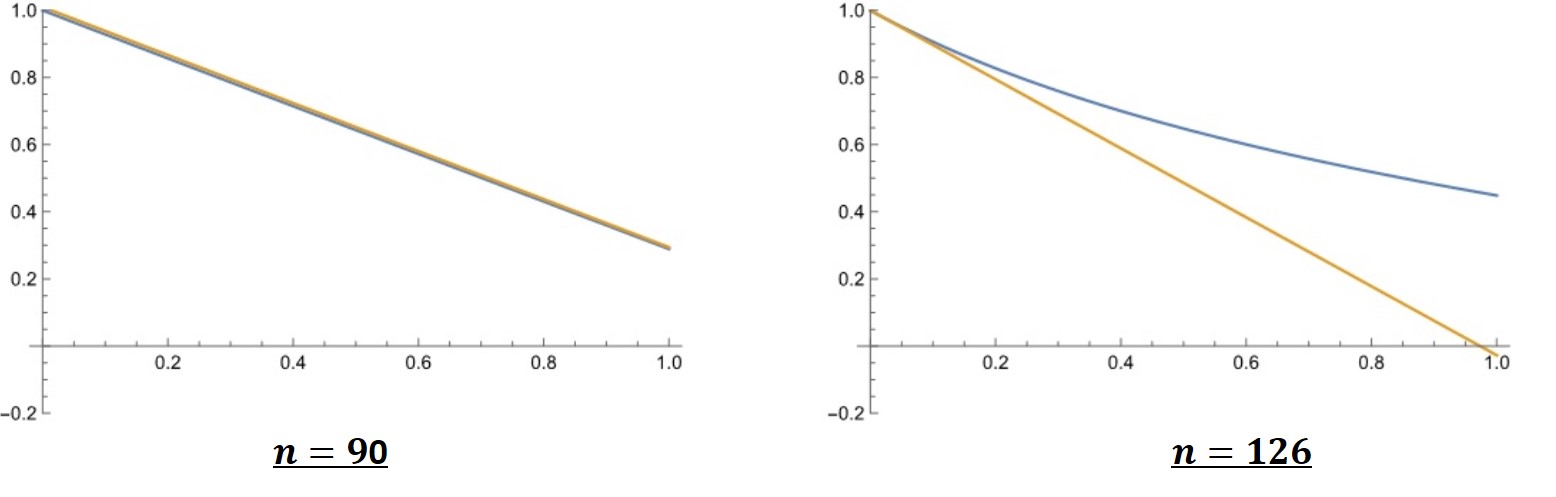} 	
		\caption{\small{Graph of $\Delta_n (r)$ (blue) and $Z_N (g_n ; r)$ (orange) for $n=90$ (left) and $n=126$ (right) with $0 \leq r \leq 1$.}}
\label{fig:f3}
	\end{figure}

	For the good Gram point $g_{90}$ the first-order linear approximation $Z_N (g_n ;r)$ serves as a rather accurate approximation of $\Delta_n (r)$ itself. For the bad Gram point $g_{126}$ the first-order approximation $Z_N (g_n ; r)$ is seen to deviate from $\Delta_n (r)$, and hence an analysis of the information encoded in the higher-order, non-linear, terms is required for such points.

 \section{The Second-Order Approximation of $\Delta_n(\overline{a})$} \label{s:6}

We now consider the second-order approximation of $\Delta_n (\overline{a})$, which in view of Theorem \ref{thm:first-order} can be written as 
\begin{equation}
\Delta_n(\overline{a} ) = Z(g_n ; \overline{a} ) +\frac{1}{2}\overline{a}^T \text{Hess}_n(0) \cdot \overline{a} +O(\abs{\overline{a}}^3),
\end{equation}
where 
\begin{equation}
\text{Hess}_n(\overline{0}) = \begin{pmatrix}
\frac{\partial^2 \Delta_n}{\partial a_1^2}(\overline{0}) & \cdots & \frac{\partial^2 \Delta_n}{\partial a_1 \partial a_N}(\overline{0}) \\
\vdots & \ddots & \vdots \\
\frac{\partial^2 \Delta_n}{\partial a_N \partial a_1}(\overline{0}) & \cdots & \frac{\partial^2 \Delta_n}{\partial a_N^2}(\overline{0})
\end{pmatrix}
\end{equation} 
is the Hessian matrix of second derivatives. The main result of this section shows the content of the second-order Hessian:
\begin{theorem}[Second-order approximation]  \label{thm:CM} For any $n \in \mathbb{Z}$, the second-order Hessian is given by 
\begin{equation} \label{eq:Hess} 
H_n(\overline{a}) :=\overline{a}^T \text{Hess}_n(0) \cdot \overline{a} =2(-1)^n \left (  \frac{Z'(g_n ; \overline{a} )}{\ln \left ( \frac{g_n }{2 \pi } \right )}  \right )^2.
\end{equation}
\end{theorem}
In order to prove the main theorem let us first prove a few preliminary results: 
\begin{lemma} 
	\label{Lem:7.2}
	\begin{equation} 
\frac{\partial}{\partial a_k} g_n(\overline{a})=  \frac{\sin ( \theta (g_n(\overline{a})) - \ln(k+1) g_n(\overline{a}))}{2 \sqrt{k+1}  Z''(g_n(\overline{a}) ; \overline{a}) } \ln \left (\frac{g_n(\overline{a})}{2 \pi (k+1)^2} \right ).
\end{equation}
	In particular, 
	\begin{equation}
\frac{\partial}{\partial a_k} g_n(\overline{0})=  2 (-1)^{n+1} \frac{\sin ( \theta (g_n) - \ln(k+1) g_n)}{ \sqrt{k+1}  \ln^2 \left ( \frac{g_n}{2 \pi} \right ) } \ln \left (\frac{g_n}{2 \pi (k+1)^2} \right ). 
\end{equation}
	\end{lemma} 

\begin{proof} 

Let $g_n(\overline{a}; \epsilon)$ 
be the $n$-th extremal point of 
\begin{equation}
F_{k,\epsilon}(t;\overline{a}):= Z_N (t ; \overline{a}) + \frac{\epsilon}{\sqrt{k+1}} \cos(\theta(t)- \ln(k+1)t).
\end{equation} 
for $0<\epsilon$ small enough. That is, the zero of the equation 
\begin{equation}
G_{k,\epsilon}(t;\overline{a}):= \frac{\partial}{\partial t} F_{k,\epsilon}(t; a)=0. 
\end{equation} 
Then according to Newton's method, one can take the following first iteration
\begin{equation}
\widetilde{g}_n(\overline{a};\epsilon) := g_n (\overline{a}) - \frac{G_{k,\epsilon}(g_n(\overline{a}) ; \overline{a})}{G_{k,\epsilon}'(g_n(\overline{a}) ;\overline{a})}, 
\end{equation} 
as an approximation of $g_n(\overline{a};\epsilon)$, which improves as $\epsilon$ decreases, see \cite{SM}. Note that 
\begin{equation}
G_{k,\epsilon}(t;\overline{a})=Z'_N (t ; \overline{a}) - \frac{\epsilon }{\sqrt{k+1}} \sin(\theta(t)- \ln(k+1)t)(\theta'(t)-\ln(k+1)). 
\end{equation}
Since $Z'_N (g_n (\overline{a}) ; \overline{a})=0$ and 
\begin{equation}
\theta'(t) = \left ( \frac{t}{2} \ln \left ( \frac{t}{2 \pi} \right ) - \frac{t}{2} -\frac{\pi}{8} \right) ' = \frac{1}{2} \ln \left ( \frac{t}{2 \pi} \right ),
\end{equation}
we have 
\begin{equation}
G_{k,\epsilon}(g_n(\overline{a}) ;\overline{a})= - \frac{\epsilon}{2\sqrt{k+1}} \sin(\theta(g_n (\overline{a}) )- \ln(k+1)g_n (\overline{a}) )\cdot  \ln \left ( \frac{g_n( \overline{a})}{2 \pi (k+1)^2} \right ). 
\end{equation}
For the derivative the main term is given by 
\begin{equation}
G'_{k,\epsilon}(g_n ( \overline{a}) ;\overline{a})= Z''_N(g_n(\overline{a})  ; \overline{a})+ O(\epsilon).
\end{equation} 
In particular, we have 
\begin{equation}
Z''_N(g_n;\overline{0})= -\cos(\theta(g_n)) (\theta'(g_n) )^2 =\frac{(-1)^{n+1}}{4} \ln^2 \left ( \frac{g_n}{2 \pi} \right ),
\end{equation}
as required.
\end{proof} 

We have: 

\begin{prop} \label{prop:7.2} For any $1 \leq k_1,k_2 \leq N$ the following holds: 
\begin{equation}
\frac{\partial^2 \Delta_{n}}{\partial a_{k_1} \partial a_{k_2}} (\overline{a}) =-\frac{1}{4 Z''(g_n(\overline{a}) ; \overline{a}) } \prod_{i=1}^2 
\frac{\sin(\theta(g_n(\overline{a}))- \ln(k_i+1) g_n (\overline{a})) \cdot \ln \left ( \frac{g_n (\overline{a})}{2 \pi (k_i+1)^2 } \right )}{\sqrt{k_i+1} } .
\end{equation}
In particular,
\begin{equation}
\frac{\partial^2 \Delta_{n}}{\partial a_{k_1} \partial a_{k_2}} (\overline{0}) =\frac{(-1)^n}{\ln^2 \left ( \frac{g_n }{2 \pi } \right )} \prod_{i=1}^2 
\frac{\sin(\ln(k_i+1) g_n) \cdot \ln \left ( \frac{g_n }{2 \pi (k_i+1)^2 } \right )}{\sqrt{(k_i+1)} } .
\end{equation} 
\end{prop} 

\begin{proof}
Consider the function 
\begin{multline} 
F_{k_1,k_2,\epsilon_1,\epsilon_2}(t;\overline{a}):= Z_N(t ; \overline{a}) + \frac{\epsilon_1}{\sqrt{k_1+1}} \cos(\theta(t)- \ln(k_1+1)t)+ \\ +\frac{\epsilon_2}{\sqrt{k_2+1}} \cos(\theta(t)- \ln(k_2+1)t)
\end{multline} 
and set 
\begin{equation} 
G_{k_1,k_2,\epsilon_1,\epsilon_2}(t;\overline{a}):= \frac{\partial}{\partial t} F_{k_1,k_2,\epsilon_1,\epsilon_2}(t; \overline{a}). 
\end{equation} 
Denote by $g_n(\overline{a}; \epsilon_1,\epsilon_2)$ the extremal point of $F_{k_1,k_2}(t;a_1,a_2)$ locally extending the gram point $g_n(\overline{a})$. Then, by definition, the discriminant can be written as 
\begin{equation}
\Delta_n( a_1,...,a_{k_1}+\epsilon_1,...,a_{k_2}+\epsilon_2,..,a_N) = F_{k_1,k_2,\epsilon_1,\epsilon_2}(g_n(\overline{a}; \epsilon_1, \epsilon_2); \overline{a} ). 
\end{equation} 
Hence, by the chain rule, we have 
\begin{multline} 
\frac{\partial \Delta_{n}}{\partial \epsilon_1} ( a_1,...,a_{k_1}+\epsilon_1,...,a_{k_2}+\epsilon_2,..,a_N) =\frac{\partial}{\partial \epsilon_1} F_{k_1,k_2,\epsilon_1,\epsilon_2}(g_n(\overline{a}; \epsilon_1, \epsilon_2); \overline{a} )=\\= 
\frac{\partial g_n}{\partial \epsilon_1}(\overline{a}; \epsilon_1,\epsilon_2) \cdot G_{k_1,k_2,\epsilon_1,\epsilon_2}(g_n(\overline{a};\epsilon_1,\epsilon_2) ; \overline{a})+ \frac{\partial}{\partial \epsilon_1} F_{k_1,k_2,\epsilon_1,\epsilon_2} (g_n(\overline{a}; \epsilon_1,\epsilon_2) ; \overline{a})=\\=\frac{\partial}{\partial \epsilon_1} F_{k_1,k_2,\epsilon_1,\epsilon_2} (g_n(\overline{a}; \epsilon_1,\epsilon_2) ; \overline{a})=\\ 
=\frac{1}{\sqrt{k_1+1}} \cos(\theta(g_n(\overline{a};\epsilon_1,\epsilon_2))- \ln(k_1+1) g_n(\overline{a}; \epsilon_1,\epsilon_2)). 
\end{multline} 
Again, we use the fact that $g_n(\overline{a}; \epsilon_1,\epsilon_2)$ are, by definition, the solutions of 
\begin{equation}
G_{k_1,k_2,\epsilon_1,\epsilon_2}(g_n(\overline{a}; \epsilon_1,\epsilon_2);\overline{a})=0.
\end{equation} 
Thus, for the second derivative we have 
\begin{multline} 
\frac{\partial^2 \Delta_{n}}{\partial \epsilon_1 \partial \epsilon_2} ( a_1,...,a_{k_1}+\epsilon_1,...,a_{k_2}+\epsilon_2,..,a_N)   =\\ 
=\frac{\partial}{\partial \epsilon_2} \left ( \frac{1}{\sqrt{k_1+1}} \cos(\theta(g_n(\overline{a};\epsilon_1,\epsilon_2))- \ln(k_1+1) g_n(\overline{a}; \epsilon_1,\epsilon_2)) \right )= \\ 
= -\frac{1}{2\sqrt{k_1+1}} \sin(\theta(g_n(\overline{a};\epsilon_1,\epsilon_2))- \ln(k_1+1) g_n(\overline{a}; \epsilon_1,\epsilon_2)) \cdot \ln \left ( \frac{g_n(\overline{a}; \epsilon_1,\epsilon_2)}{2 \pi (k+1)^2 } \right )  \frac{\partial g_n}{\partial \epsilon_2} (\overline{a} ; \epsilon_1,\epsilon_2). 
\end{multline} 
By substituting $(\epsilon_1,\epsilon_2)=(0,0)$ and applying Lemma \ref{Lem:7.2} the result follows.   
\end{proof} 
\begin{proof} [Proof of Theorem \ref{thm:CM}:] By Proposition \ref{prop:7.2} we have 
\begin{multline}
H_n(\overline{a}):=\sum_{k_1,k_2=1}^N \frac{\partial^2 \Delta_{n}}{\partial a_{k_1} \partial a_{k_2}} (\overline{0}) \cdot a_{k_1} a_{k_2},
=\\= \frac{(-1)^n}{\ln^2 \left ( \frac{g_n }{2 \pi } \right )} \sum_{k_1,k_2=1}^N  \prod_{i=1}^2 
\frac{\sin(\ln(k_i+1) g_n) \cdot \ln \left ( \frac{g_n }{2 \pi (k_i+1)^2 } \right )}{\sqrt{k_i+1} } a_{k_i}=\\=
\frac{(-1)^n}{\ln^2 \left ( \frac{g_n }{2 \pi } \right )} \left ( \sum_{k=1}^N  
\frac{\sin(\ln(k+1) g_n) \cdot \ln \left ( \frac{g_n }{2 \pi (k+1)^2 } \right )}{\sqrt{k+1} } a_{k}\right )^2=
4(-1)^n \left (  \frac{Z'(g_n; \overline{a})}{\ln \left ( \frac{g_n }{2 \pi } \right )}  \right )^2,
\end{multline} 
as required. 
\end{proof} 

It is crucial to note that while adding higher-order terms would improve the accuracy of the approximation near \(\overline{a}=\overline{0}\), the accuracy around \(\overline{a}=\overline{1}\), where our primary interest lies, can actually diminish. In particular, obtaining a reasonable approximation of \(\Delta_n(\overline{1})\) via Taylor expansion might require an impractically large number of higher-degree terms, especially for general \(n \in \mathbb{Z}\). Therefore, in the following sections, we will focus on further analysing the substantial mathematical information already encoded in our second-order approximation.

\section{The Second-Order Hessian and Shifts of Gram Points Along the $t$-Axis}
\label{s:6.5}

Figure \ref{fig:f4} shows the graphs of $Z_N(t;r)$ in the range $t \in [g_n -2,g_n+2]$ for $n=90$ (left) and $n=126$ (right) and various values of $r \in [0,1]$.
\begin{figure}[ht!]
	\centering
		\includegraphics[scale=0.35]{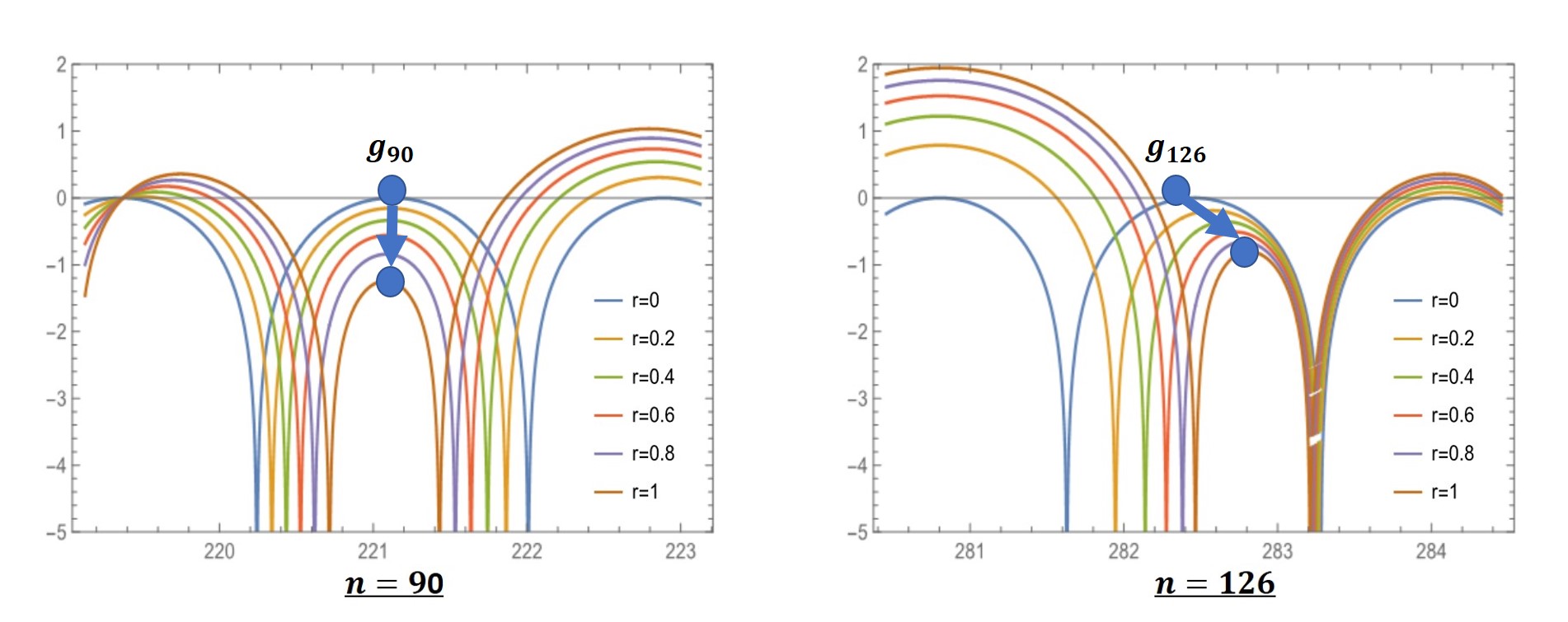} 	
		\caption{\small{Graphs of $ \ln \left |Z_N(t;r) \right |$ in the range $t \in [g_n -2,g_n+2]$ for $n=90$ (left) and $n=126$ (right) and various values of $r \in [0,1]$.}}
\label{fig:f4}
	\end{figure} 
	
Note that these are the same two Gram points whose discriminant \(\Delta_n(r)\) is presented in Fig. \ref{fig:f3}. Figure \ref{fig:f4} illustrates that for the good Gram point \(g_{91}\), whose discriminant is well approximated by the first-order approximation $Z(g_{91}; r)$, the position of \(g_{91}(r)\) experiences minimal change as \(r\) increases. Conversely, for the bad Gram point \(g_{126}\), the decline in the first-order approximation \(Z(g_{126}; r)\) is accompanied by a significant shift in \(g_{126}(r)\), as well. By direct computation, we have  
\begin{equation} 
\begin{array}{ccc} H_{90}(\overline{1}) =0.00203615 & ; & H_{126}(\overline{1}) = 2.22893. \end{array} 
\end{equation}  
Remarkably, this phenomena is observed to hold in great generality, as will be presented in the next section. Our aim in this section is to prove the correlation between $H_{n}( \overline{a})$ and shifts of $g_{n}(\overline{a})$ along the $t$-axis. Indeed, according to Theorem \ref{thm:CM} the magnitude of the $H_n (\overline{a})$ is proportional to $(Z'(g_n; \overline{a}))^2$. The following result shows the variational content of the value of $Z'(g_n ; \overline{a})$: 
\begin{theorem} \label{thm:Grad}
For any $n \in \mathbb{Z}$ the following holds
\begin{equation} 
Z'(g_n ; \overline{a} ) = \frac{1}{4} (-1)^n \ln^2 \left ( \frac{g_n}{2 \pi} \right ) \overline{a} \cdot \nabla g_n(0), 
\end{equation}
where 
\begin{equation}
\nabla  g_{n} (\overline{0}) := \left (  \frac{\partial g_{n}}{\partial a_1} (\overline{0}) ,...,\frac{\partial g_{n}}{\partial a_N} (\overline{0}) \right )
\end{equation}
is the gradient of $g_n(\overline{a})$ at $ \overline{a}=\overline{0}$.
\end{theorem} 
\begin{proof} 
The following holds 
\begin{multline} 
Z'_N(t; \overline{a}) = -\sin(\theta(t))\theta'(t) -\sum_{k=1}^{N} \frac{a_k}{\sqrt{k+1} } \sin ( \theta (t) - \ln(k+1) t) (\theta'(t)-\ln(k+1)) = 
\\=  -\frac{1}{2} \sin(\theta(t))\ln \left ( \frac{t}{2 \pi} \right )  -\sum_{k=1}^{N} \frac{a_k}{2\sqrt{k+1} } \sin ( \theta (t) - \ln(k+1) t)  \ln \left ( \frac{t}{2 \pi (k+1)^2} \right ). 
\end{multline} 
Hence, 
\begin{equation}
Z'_N(g_n ; \overline{a})= \sum_{k=1}^{N} \frac{a_k}{2\sqrt{k+1} } \sin (\ln(k+1) g_n)  \ln \left ( \frac{g_n}{2 \pi (k+1)^2} \right ). 
\end{equation} 
But also, 
\begin{equation}
\frac{1}{4} \ln^2 \left ( \frac{g_n}{2 \pi} \right ) \frac{\partial}{\partial a_k} g_n(\overline{0})=   (-1)^{n} \frac{\sin ( \ln(k+1) g_n)}{ 2 \sqrt{k+1} } \ln \left (\frac{g_n}{2 \pi (k+1)^2} \right ). 
\end{equation} 
Hence, 
\begin{equation}
Z'(g_n ; \overline{a} ) = \frac{1}{4} (-1)^n \ln^2 \left ( \frac{g_n}{2 \pi} \right ) \overline{a} \cdot \nabla g_n(0).
\end{equation} 
\end{proof} 
Theorem \ref{thm:CM} and Theorem \ref{thm:Grad} together imply the following: 

\begin{cor} \label{Hess-grad} The following holds:  
\begin{enumerate} 
\item  The second-order Hessian $H_n ( \overline{1})$ measures the magnitude of the gradient $\nabla g_n(0)$. 
\item The direction of the shift of the $n$-th extremal point of $Z(t; \overline{a})$ with respect to $g_n$ is given by the sign of 
$(-1)^n Z'(g_n ; \overline{a})$.   
\end{enumerate}
\end{cor} 

In view of the RH, the corrected Gram's law implies that for bad Gram points the second order term, represented by  $H_n(\overline{1})$, is expected to become crucial in order to compensate as a correcting term on the first-order violation of the classical law. Corollary \ref{Hess-grad} further shows that a large $H_n (\overline{1})$ is expressed by the fact that the variation of the Gram point, $g_n(\overline{a})$, must experience a considerable positional shift for the approximation to be valid. Conversely, a small $H_n (\overline{1})$ signifies that the first order term predominates, and the position of the variation of the Gram point $g_n(\overline{a})$ undergoes relatively small change. In the next section we describe how the results of this section lead to the discovery of a new repulsion phenomena between pairs of consecutive zeros of $Z(t)$, and the relation of this newly discovered property to the classical Montgomery pair correlation conjecture.

 \section{A New Dynamic Repulsion Property Between Consecutive Zeros}
 \label{s:7} 

The results of the previous section imply that if $g_n$ is a bad Gram point, one anticipates a significant shift in the position of $g_n(r)$ itself to fulfil the corrected Gram's law at the second-order level. In other words, bad Gram points $g_n$ are expected to demonstrate some correlation between the values of $Z(g_n)$ and $Z'(g_n)$. In simple terms, for a bad Gram point, we dynamically interpret $Z(g_n)$ as pushing the extremal point $g_n(r)$ downward towards zero to create a collision, while $Z'(g_n)$ pushes the extremal point sideways to avoid a collision. Thus, let us introduce the following definition:
 
 \begin{dfnnonum}[Viscosity of a Gram point] For any $n \in \mathbb{Z}$ we refer to\footnote{
 Recall that in general $
\frac{Z'(t)}{Z(t)}=i \theta'(t) -\frac{1}{it-\frac{1}{2}}+\sum_{\rho} \frac{1}{\frac{1}{2}+it-\rho}- \frac{1}{2} \frac{\Gamma'(\frac{5}{4}+\frac{it}{2})}{\Gamma(\frac{5}{4}+\frac{it}{2})}+\frac{1}{2}\ln(\pi),$ see  3.2 of \cite{E}. }
\begin{equation}
\mu(g_n) := \left \vert \frac{Z'(g_n)}{Z(g_n)} \right \vert
\end{equation} 
as \emph{the viscosity of the Gram point $g_n$.}
 \end{dfnnonum}
 
 The viscosity of a Gram point $\mu(g_n)$ is essentially a measure of how much a Gram point 'resists' a change in its position, analogous to how viscosity in a fluid quantifies its resistance to flow. In essence, a high viscosity at a Gram point implies a lower tendency for the point to maintain its position and vice versa.

Let us denote by \(g_n^{\text{bad}}\) the subsequence of bad Gram points among the Gram points \(g_n\) (i.e., \(g_1^{\text{bad}} = g_{126}\), \(g_2^{\text{bad}} = g_{134}\), etc.).  In view of the above, we anticipate that the viscosity, $\mu(g^{\text{bad}}_n)$, of bad Gram points will exhibit unique features compared to the viscosity $\mu(g_n)$ of general general Gram points. Figure \ref{fig:f51} shows the viscosity \(\mu(g_n)\) and $\mu(g_n^{\text{bad}})$ for the first \(n=1,\ldots,1000\) general Gram points (left) and bad Gram points (right). 
\begin{figure}[ht!]
	\centering
		\includegraphics[scale=0.35]{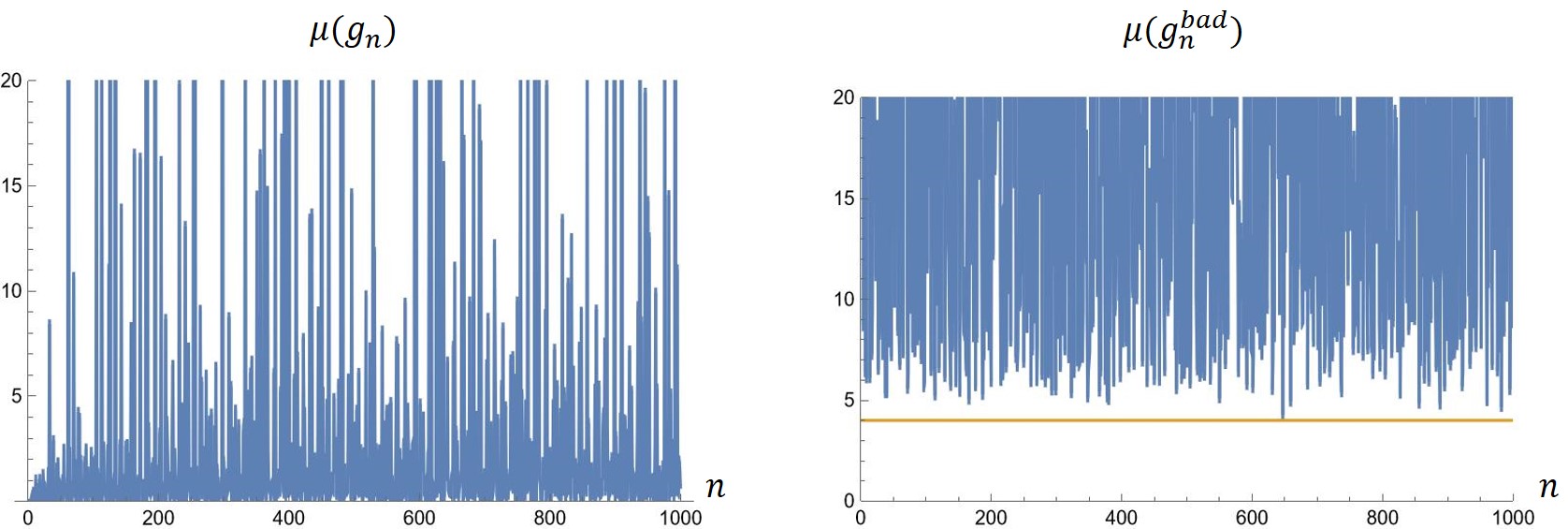} 	
		\caption{\small{Viscosity \(\mu(g_n)\) and $\mu(g_n^{\text{bad}})$ for the first \(n=1,\ldots,1000\) general Gram points (left) and bad Gram points (right).}}
\label{fig:f51}
\end{figure}

Remarkably, Fig. \ref{fig:f51} shows that while the values of \(\mu(g_n)\) for general Gram points appear to be distributed without a discernible pattern or lower bound, the values for bad Gram points \(\mu(g_n^{\text{bad}})\) seem to be bounded from below by a constant \(C > 4\), at least for the first thousand points presented.

In fact, another layer of complexity is encountered as the computational range is extended further. While the phenomena observed in Fig. \ref{fig:f51} is seen to persist for far wider ranges, as we proceed to higher values, a sparse subset of bad Gram points is discovered to sporadically defy the bound, yielding bad gram points with viscosity dramatically below $C$. We will refer to such unusual instances of bad Gram points as \emph{corrupt} Gram points. For instance, the 9807962-th bad Gram point is corrupt and its viscosity is 
\begin{equation}
\label{eq:corrupt}
\mu(g_{9807962})=0.0750883.
\end{equation}

Figure \ref{fig:f523} shows the viscosity \(\mu(g^{bad}_n)\) of all the bad Gram points arising between the $2.4 \cdot 10^7$-th and $2.43 \cdot 10^7$-th Gram points, with the corrupt points marked red: 
\begin{figure}[ht!]
	\centering
		\includegraphics[scale=0.275]{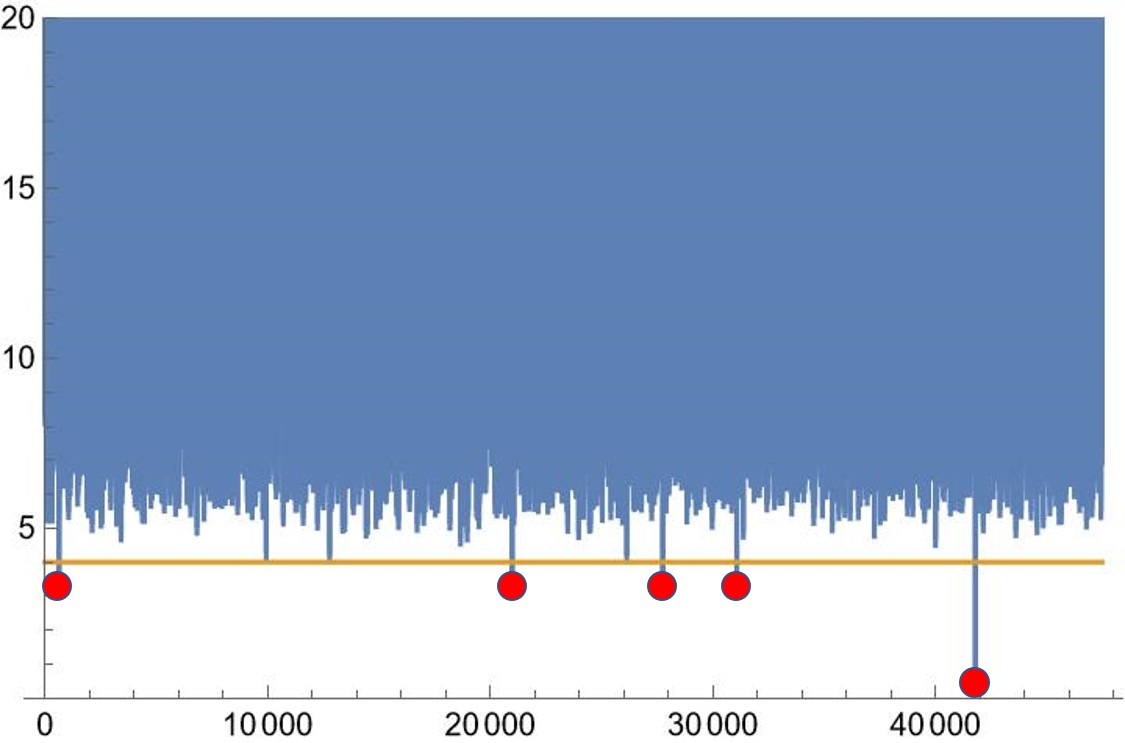} 	
		\caption{\small{Viscosity \(\mu(g^{bad}_n)\) of the bad Gram points between the $2.4 \cdot 10^7$-th and $2.43 \cdot 10^7$-th Gram points with corrupt points marked red.}}
\label{fig:f523}
\end{figure}

The existence of corrupt Gram points may seem as a conclusive counter-example for the general validity of the bound observed in Fig. \ref{fig:f51}. However, it turns out that corrupt Gram points exhibit a distinctive characteristic and are observed to occur only under very specific conditions. Let us recall the following classic definition due to \cite{Ro}:
\begin{dfn}[Gram block]
A consecutive collection $\left \{ g_n,g_{n+1},...,g_{n+N} \right \}$ of Gram points is called a \emph{Gram block} if $g_n$ and $g_{n+N}$ are good Gram points while $g_{n+j}$ are bad Gram points for $j=1,...,N-1$. We refer to a bad Gram point as \emph{isolated} if it is the middle point of a block with $N=2$.  
\end{dfn} 	
For instance, the corrupt Gram point \eqref{eq:corrupt} from the example above is part of the following Gram block of length $N=3$
\begin{equation} 
\left \{ 9807960,9807961,9807962,9807963 \right \}.
\end{equation}
Based on vast numerical verifications, extending substantially beyond the sample examples presented here in Fig. \ref{fig:f51} and Fig. \ref{fig:f523}, we are led to conjecture: 

\begin{conj}[Repulsion for isolated bad Gram points] \label{con:8.1} If $g_n$ is an isolated bad Gram point then 
\begin{equation} \label{eq:rep2}
\abs{Z'(g_n)}>4\abs{Z(g_n)}.
\end{equation} 
In particular, corrupt Gram points must be non-isolated.  
\end{conj}

 The relation \eqref{eq:rep2} implies a correlation between the value of $\abs{Z(g_n)}$ and $\abs{Z'(g_n)}$ for isolated bad Gram points. Consequently, if there exists a force pushing the value of \(Z(t)\) at \(g_n\) towards the axis, there must also be a lateral force pushing it sideways. This suggests that, on an infinitesimal level, for  isolated bad Gram points, the discriminant $\Delta_n(r)$ resist a change in sign. Given that a change in sign in $\Delta_n(r)$ is equivalent to the occurrence of a collision between two consecutive zeros $t_n(r)$ and $t_{n+1}(r)$, the bound implies a newly discovered mechanism of consecutive zeros of the $Z$-function to resist collision between each other,  or a repulsion between consecutive zeros. 
 
Moreover, to the best of our knowledge, this is the first property in the literature observed to be distinctly satisfied by bad Gram points, uniquely distinguishing them from good ones. Before proceeding to discuss the possible implications of the repulsion relation to the study of RH, let us make the following important remark:

\begin{rem}[Montgomery's Conjecture - Statistical vs. Dynamical Repulsion]  Due to the decay of the integral \eqref{eq:pcint} for small \(u\), the Montgomery pair correlation conjecture is often interpreted as anticipating the existence of a repulsion phenomena between consecutive zeros of $Z(t)$. However, it is crucial to emphasize that this notion of repulsion is based on the probabilistic properties of zeros in the large scale. In particular, while the term repulsion hints on an underlying dynamical phenomena, in the setting of Montgomery's conjecture the $Z$-function as well as its zeros are actually static and the repulsion is meant merely in a statistical sense. 

In contrast, in our setting, paths in the space $\mathcal{Z}_N$ lead to a genuine notion of continuous variations of zeros $t_n(r)$ and the repulsion relation \eqref{eq:rep2} is hence a dynamic property of consecutive pair of zeros rather than a statistical property regarding their distribution. In fact, we argue that the newly discovered dynamic repulsion relation \eqref{eq:rep2} might in practice be the phenomena responsible for the statistical repulsion anticipated by Montgomery's conjecture.  
\end{rem} 
 
\section{Repulsion and the RH - Revisiting Edwards' Speculation} \label{s:7.5}
According to Theorem \ref{thm:ESHO}, the RH is equivalent to finding a non-colliding curve \(\gamma_n(r)\) for any \(n \in \mathbb{Z}\) such that 
\begin{equation} 
\label{eq:pos}
(-1)^n \Delta_n (\gamma_n(r)) > 0,
\end{equation} 
for \(0 \leq r \leq 1\), and connecting \(Z_0(t)\) to \(Z(t ; \overline{1})\) in \(\mathcal{Z}_N\). For good Gram points, as well as for many bad Gram points, the linear curve $\overline{r} = (r,...,r)$ is actually sufficient and satisfies \eqref{eq:pos}, as seen for instance in Fig. \ref{fig:f3}. 

However, bad Gram points \( g_n \) for which \( (-1)^n \Delta_n(r) \) is not consistently positive, are also observed to exist. For instance, consider $g_{730119}$ which is an isolated bad Gram point with relatively very small viscosity $\mu(g_{730119}) \approx 4.4602$. Figure \ref{fig:f8.2} shows the graphs of $\Delta_{730119} (r)$ (blue) and $Z_N (g_{730119} ; r)$ (orange) with $0 \leq r \leq 1$:
 \begin{figure}[ht!]
	\centering
		\includegraphics[scale=0.4]{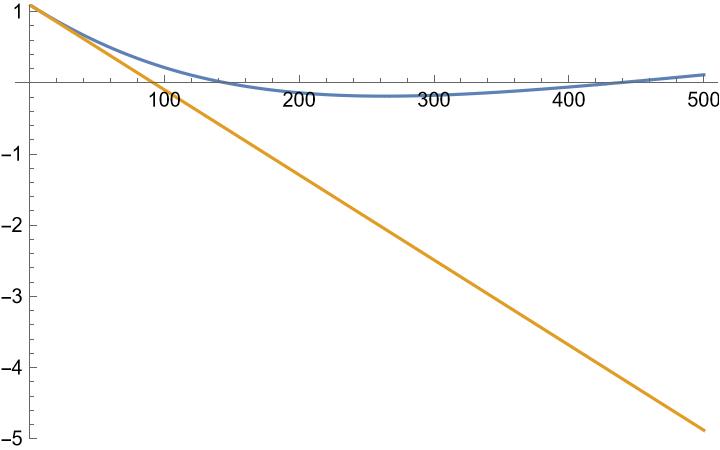} 	
		\caption{\small{Graph of the discriminant $-\Delta_{730119} (r)$ (blue) and $-Z_N (g_{730119} ; r)$ (orange) with $0 \leq r \leq 1$.}}
\label{fig:f8.2}
	\end{figure}
	
Contrary to the previous examples of Fig. \ref{fig:f3}, the discriminant $\Delta_{730119}(r)$ is not positive for all $0 \leq r \leq 1$, and there is a region where it attains negative values. Hence, the RH is in essence the question of whether one can find an alternative $\gamma_n(r)$ for the linear curve satisfying \eqref{eq:pos}, for such Gram points. Figure \ref{fig:f9} presents a schematic representation of the $A$-variation space and the relevant objects in such cases: 
\begin{figure}[ht!]
	\centering
		\includegraphics[scale=0.3]{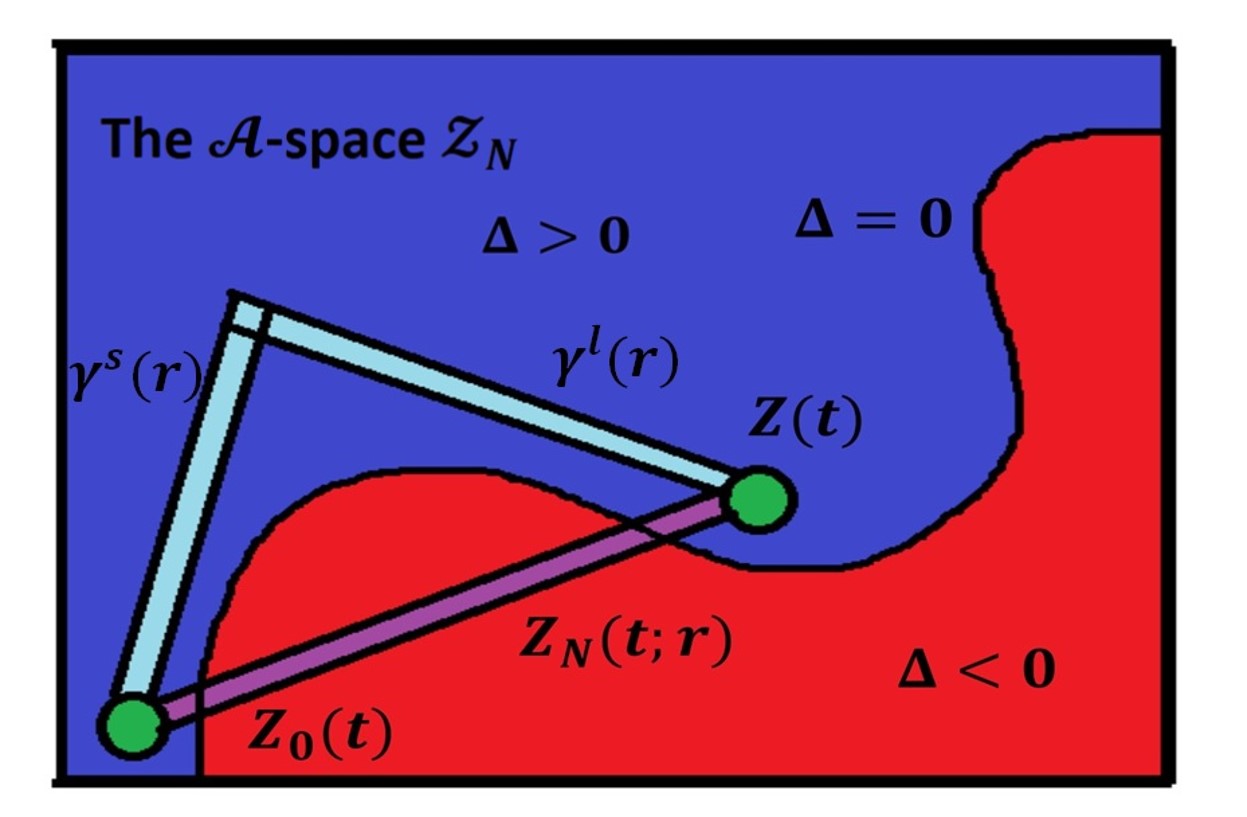} 	
		\caption{\small{Illustration of the linear curve $\overline{r}$ (purple) and a non-colliding alternative $\gamma_n(r)$ (cyan) connecting the core $Z_0(t)$ to $Z_N(t; \overline{1}) \approx Z(t)$ in the space $\mathcal{Z}_N$.}}
\label{fig:f9}
	\end{figure}

In \cite{J5} we identified a non-colliding curve $\gamma_n(r)$ for $n=730119$ and described it in length. Consider the $2$-parametric system 

\begin{multline} \label{eq:split}
Z_N(t;r_1,r_2):= Z_0(t) +  \sum_{k \in I_{shift}} \frac{r_1}{\sqrt{k+1}} cos(\theta(t)-ln(k+1)t) + \\+  \sum_{k \in I_{collide}}  \frac{r_2}{\sqrt{k+1} }cos(\theta(t)-ln(k+1)t)  \in \mathcal{Z}_N,
\end{multline} 
arising from splitting of the indices into the shifting indices 
\begin{equation} 
\label{eq:split}
I_{shift}= \left \{ 1,2,4,6,12 \right \}, 
\end{equation} 
and colliding indices $I_{collide}$, taken to be the complement of $I_{shift}$. The curve $\gamma_n(r)$ was defined as a concatenation of $\gamma^1_n(r)$, starting from $(r_1,r_2)=(0,0)$ and approaching $(r_1,r_2)=(1,0.41)$, with $\gamma^2_n(r)$ starting from $(r_1,r_2)=(1,0.41)$ and approaching $(r_1,r_2)=(1,1)$.  

In \cite{J5}, the split of the indices \eqref{eq:split} into shifting and colliding indices was obtained experimentally. Our aim now is to explain this split in terms of the repulsion relation. Recall that according to the classical AFE,  for any \( n \in \mathbb{Z} \) the following holds:
\begin{equation}
\begin{array}{ccc}
    Z(g_n) =\sum_{k=1}^{N(n)} A_k(g_n) + O\left( g_n^{-\frac{1}{4}} \right) & ; &
    Z'(g_n) = \sum_{k=1}^{N(n)} B_k(g_n)  + O\left( g_n^{-\frac{1}{4}} \right),
\end{array}
\end{equation}
where 
\begin{equation}
\left\{
\begin{aligned}
    &A_k(g_n) = 2 (-1)^n \frac{\cos(\ln(k+1) g_n)}{\sqrt{k+1}} , \label{eq:Zgn} \\
    &B_k(g_n) = (-1)^n  \ln\left(\frac{g_n}{2\pi (k+1)^2}\right) \frac{\sin(\ln(k+1) g_n)}{\sqrt{k+1}},
\end{aligned}
\right.
\end{equation}
and \( N(n) := \left [ \sqrt{\frac{g_n}{2\pi}} \right ] \), see for instance (5.2) and (6.3) in \cite{I}. Table \ref{tab:values} shows the values of $\cos(\ln(k)g_n), \sin(\ln(k)g_n), A_k(g_n)$ and $B_k(g_n)$ within the range $k=1,...,\sqrt{N(n)}$: 
\captionsetup{skip=10pt}

\definecolor{LightRed}{rgb}{1,0.8,0.8}
\begin{table}[h!]
    \centering
    \resizebox{\textwidth}{!}{%
        \begin{tabular}{
            |c|
            >{\columncolor{red!20}}c|
            >{\columncolor{red!20}}c|
            c|
            >{\columncolor{red!20}}c|
            c|
            >{\columncolor{red!20}}c|
            c|c|c|c|c|
            >{\columncolor{red!20}}c|
            c|c|c|
        }
            \arrayrulecolor{black}\hline
            \( k \) & 1 & 2 & 3 & 4 & 5 & 6 & 7 & 8 & 9 & 10 & 11 & 12 & 13 & 14 & 15 \\
            \hline
            \(\cos(\ln(k)g_n)\) & -0.14 & 0.25 & 0.96 & -0.53 & 0.99 & -0.20 & 0.41 & 0.88 & 0.77 & -0.99 & 0.03 & 0.21 & 0.94 & 0.95 & -0.85 \\
            \hline
            \(\sin(\ln(k)g_n)\) & 0.99 & 0.97 & 0.28 & 0.85 & -0.11 & 0.98 & -0.91 & -0.48 & 0.64 & -0.11 & -1.0 & 0.98 & 0.33 & 0.30 & -0.53 \\
            \hline
            \( A_k(g_n) \) & -0.099 & 0.14 & 0.48 & -0.24 & 0.41 & -0.074 & 0.14 & 0.29 & 0.24 & -0.30 & 0.0082 & 0.058 & 0.25 & 0.25 & -0.21 \\
            \hline
            \( B_k(g_n) \) & 6.86 & 5.02 & 1.16 & 3.02 & -0.345 & 2.70 & -2.27 & -1.09 & 1.34 & -0.210 & -1.79 & 1.64 & 0.521 & 0.449 & -0.748 \\
            \hline
        \end{tabular}
    }
    \caption{\small{$\cos(\ln(k)g_n), \sin(\ln(k)g_n), A_k(g_n)$ and $B_k(g_n)$ for \( k=1,...,15 \) with the shifting indices   $k=1,2,4,6,12$ marked red.}}
    \label{tab:values}
\end{table}

We see that the shifting indices $k=1,2,4,6,12$ (red) were chosen to be those for which $B_k(g_n)$ is especially large, and are marked red in the table. In the specific case of $n=730119$ the shifting indices were indeed found via direct computation. In general, one can interpret the repulsion relation \eqref{eq:rep2} as guaranteeing the existence of such a collection of shifting indices for any isolated bad Gram point and hence for the possibility to generalize the definition of the non-colliding curve $\gamma_n(r)$ for any such $g_n$, holding profound implications for the study of the RH.     

 \section{Summary and Concluding Remarks} 
 \label{s:8}
 
 The Riemann Hypothesis postulates that all the non-trivial solutions of the equation \( Z(t)=0 \) must be real. In algebraic geometry one has the powerful invariant of the discriminant, which can be seen as a measurement for the realness of zeros of algebraic equations, depending on parameter.
 
  In this work, we extended the idea of the discriminant into the transcendental setting of the \( Z(t) \) function by considering the \( A \)-parametrized space \( \mathcal{Z}_N \) whose elements are given by
\begin{equation}
Z_N(t ; \overline{a}) = \cos(\theta(t)) + \sum_{k=1}^{N} \frac{a_k}{\sqrt{k+1}} \cos \left( \theta (t) - \ln(k+1) t \right),
\end{equation}
where \( \overline{a} = (a_1, \dots, a_N) \in \mathbb{R}^N \) for any \( N \in \mathbb{N} \). We defined the local discriminant for a pair of consecutive zeros, \( \Delta_n(\overline{a}) \), as a function over the parameter space $\mathcal{Z}_{N(n)}$ of dimension \( N(n) := \left [ \frac{\abs{g_n}}{2} \right ] \), vanishing in instances when the two zeros collide to form a double root. 

This newly defined discriminant has unveiled a wealth of significant new results regarding the zeros of \(Z(t)\). For instance, we demonstrated in Theorem \ref{thm:B} that the RH is equivalent to the fulfilment of the corrected Gram's law, \( (-1)^n \Delta_n(\overline{1}) > 0 \), for any \(n \in \mathbb{Z}\). In Theorem \ref{thm:B2}, we further showed that the classical Gram's law arises as the first-order approximation of our corrected Gram's law for the linear curve \(Z_N(t; r)\) in the parameter space \(\mathcal{Z}_N\). The properties of the second-order Hessian, presented in Theorem \ref{thm:C}, revealed its relation to shifts of the Gram points along the \(t\)-axis. Based on our discriminant analysis, we discovered the previously unobserved numerical repulsion relation \( \left| Z'(g_n) \right| > 4 \left| Z(g_n) \right| \) for isolated bad Gram points \(g_n\), discussing its profound implications for both the Montgomery pair correlation conjecture and the RH.

Collectively, the introduction of $\Delta_n(\overline{a})$ together with the conjectures, experimental discoveries and theorems established in this work introduce a new dynamical and promising approach for the further study of the zeros of the $Z$-function.   

\section*{Declarations}

\subsection*{Funding}
No funding was received to assist with the preparation of this manuscript.

\subsection*{Conflicts of Interest/Competing Interests}
The authors declare that they have no conflict of interest or competing interests relevant to the content of this article.

\subsection*{Data Availability}
The authors declare that the data supporting the findings of this study are available within the paper or from the corresponding author upon reasonable request.

\bibliographystyle{plain} 
\bibliography{ref} 

\begin{thebibliography}{10}

\bibitem{Berry}
M.~V. Berry.
\newblock Semiclassical formula for the number variance of the riemann zeros.
\newblock {\em Nonlinearity}, 1:399--407, 1988.

\bibitem{BeKe}
M.~V. Berry and J.~P. Keating.
\newblock The riemann zeros and eigenvalue asymptotics.
\newblock {\em SIAM Review}, 41(2):236--266, 1999.

\bibitem{Bog2}
E.~Bogomolny.
\newblock Riemann zeta function and quantum chaos.
\newblock {\em Progress of Theoretical Physics Supplement}, (166):19, 2007.

\bibitem{Bog1}
E~Bogomolny and P~Leboeuf.
\newblock Statistical properties of the zeros of zeta functions-beyond the
  riemann case.
\newblock {\em Nonlinearity}, 7(4):1155, 1994.

\bibitem{BogKeat1}
E.~B. Bogomolny and J.~P. Keating.
\newblock Random matrix theory and the riemann zeros i: three- and four-point
  correlations.
\newblock {\em Nonlinearity}, 8:1115--1131, 1995.

\bibitem{BogKeat2}
E.~B. Bogomolny and J.~P. Keating.
\newblock Random matrix theory and the riemann zeros ii: n-point correlations.
\newblock {\em Nonlinearity}, 9:911--935, 1996.

\bibitem{Borwein2008Riemann}
P.~Borwein, S.~Choi, B.~Rooney, and A.~Weirathmueller.
\newblock {\em {The Riemann Hypothesis: A Resource for the Afficionado and
  Virtuoso Alike}}.
\newblock Springer, 2008.

\bibitem{BB}
K.A. Broughan and A.R. Barnett.
\newblock Gram lines and the average of the real part of the riemann zeta
  function.
\newblock {\em Mathematics of Computation}, pages 1--11, 2011.

\bibitem{E}
H.~M. Edwards.
\newblock {\em {Riemann's} Zeta Function.}
\newblock Academic Press., 1974.

\bibitem{FL}
G.~Franca and A.~LeClair.
\newblock {\em {Statistical and Other Properties of {Riemann} Zeros Based on an
  Explicit Equation for the N-th Zero on the Critical Line}}.
\newblock 2013.

\bibitem{Gourdon2004}
X.~Gourdon.
\newblock {{The \(10^{13}\) First Zeros of the Riemann Zeta Function, and Zeros
  Computation at Very Large Height}}, 2004.
\newblock Available online.

\bibitem{G}
J.~P. Gram.
\newblock Sur les zéros de la fonction $\zeta(s)$ de {Riemann}.
\newblock {\em Acta Math.}, 27, 1903.

\bibitem{HL}
G.~H. Hardy and J.~E. Littlewood.
\newblock {The Zeros of {Riemann's} Zeta Function on the Critical Line}.
\newblock {\em {Math. Z.}}, 10:283--317, 1921.

\bibitem{HL2}
G.~H. Hardy and J.~E. Littlewood.
\newblock {The Approximate Functional Equation in the Theory of the Zeta
  Function, with an Application to the Divisor-problems of {Dirichlet} and
  {Piltz}}.
\newblock {\em {Proceedings of the London Mathematical Society}},
  s2-21(1):39--74, 1923.

\bibitem{HL3}
G.~H. Hardy and J.~E. Littlewood.
\newblock {The Approximate Functional Equations for $\zeta(s)$ and
  $\zeta^2(s)$}.
\newblock {\em {Proceedings of the London Mathematical Society}},
  s2-29(1):81--97, 1929.

\bibitem{Hu}
J.~I. Hutchinson.
\newblock On the roots of the {Riemann} zeta-function.
\newblock {\em Trans. Amer. Math. Soc.}, 27, 1925.

\bibitem{GKZ}
M.~M.~Kapranov I.~M.~Gelfand and A.~V. Zelevinsky.
\newblock {\em Discriminants, Resultants, and Multidimensional Determinants.}
\newblock Birkhauser., 2008.

\bibitem{I}
A.~Ivić.
\newblock {\em {{The Theory of {Hardy's} {Z}-Function}}}.
\newblock Cambridge: Cambridge University Press, 2012.

\bibitem{J}
Y.~Jerby.
\newblock {{An Approximate Functional Equation for the {Riemann} Zeta Function
  with Exponentially Decaying Error}}.
\newblock {\em Journal of Approximation Theory},
  265:https://doi.org/10.1016/j.jat.2021.105551, 2021.

\bibitem{J4}
Y.~Jerby.
\newblock {{On the Approximation of the Hardy {{$Z$}}-Function via High-Order
  Sections}}.
\newblock {\em preprint}, 2024.

\bibitem{J5}
Y.~Jerby.
\newblock {{On the Edwards' Speculation and a New Variational Method for the
  Zeros of the $Z$-Function}}.
\newblock {\em preprint}, 2024.

\bibitem{Kor2}
M.A. Korolev.
\newblock Gram’s law and selberg’s conjecture on the distribution of the
  zeros of the riemann zeta-function.
\newblock {\em Izv. Ross. Akad. Nauk. Ser. Mat.}, 74(4):83--118, 2010.
\newblock In Russian. English translation in Izv. Math. 74(4), 743--780 (2010).

\bibitem{Kor1}
M.A. Korolev.
\newblock Gram’s law in the theory of the riemann zeta-function. part 1.
\newblock {\em Proc. Steklov Inst. Math.}, 292(2):1--146, 2016.

\bibitem{Le}
R.~S. Lehman.
\newblock {{On the Distribution of Zeros of the {Riemann} Zeta Function}}.
\newblock {\em Proc. London Math. Soc.}, 20:303--320, 1970.

\bibitem{KS}
K.~Nicholas M. and S.~Peter.
\newblock Zeroes of zeta functions and symmetry.
\newblock {\em Bulletin of the American Mathematical Society, New Series},
  36(1):1--26, 1999.

\bibitem{M}
H.~L. Montgomery.
\newblock {\em The pair correlation of zeros of the zeta function.}
\newblock Analytic number theory, Proc. Sympos. Pure Math., XXIV, Providence,
  R.I.: American Mathematical Society,, 1973.

\bibitem{Odlyzko1992}
A.~Odlyzko.
\newblock {{The \(10^{20}\)-th Zero of the Riemann Zeta Function and 175
  Million of Its Neighbors}}.
\newblock Unpublished manuscript, 1992.

\bibitem{Od}
A.~M. Odlyzko.
\newblock {{On the Distribution of Spacings Between Zeros of the Zeta
  Function}}.
\newblock {\em Mathematics of Computation}, 48(177):273--308, 1987.

\bibitem{OdlyzkoSchoenhage1988}
A.~M. Odlyzko and A.~Schönhage.
\newblock {{Fast Algorithms for Multiple Evaluations of the Riemann Zeta
  Function}}.
\newblock {\em Trans. Amer. Math. Soc.}, 309(2):797--809, 1988.

\bibitem{PT}
D.~Platt and T.~Trudgian.
\newblock {{The {Riemann} Hypothesis is True up to $3\cdot10^{12}$}}.
\newblock {\em Bulletin of the London Mathematical Society, Wiley}, page
  doi:10.1112/blms.12460, 2021.

\bibitem{R}
B.~Riemann.
\newblock {{Über die Anzahl der Primzahlen unter einer gegebenen Grösse}}.
\newblock {\em Monatsber. Königl. Preuss. Akad. Wiss. Berlin, 671-680, Nov.
  1859. Reprinted in Das Kontinuum und Andere Monographen (Ed. H. Weyl). New
  York: Chelsea}, 1972.

\bibitem{Ro}
J.~B. Rosser, J.~M. Yohe, and L.~Schoenfeld.
\newblock {{Rigorous Computation and the Zeros of the {Riemann}
  Zeta-Function}}.
\newblock In {\em Cong. Proc. Int. Fed. Information Process.}, pages 70--76,
  Washington, DC, 1969. Spartan.

\bibitem{RS}
Z.~Rudnick and P.~Sarnak.
\newblock Zeros of principal l-functions and random matrix theory.
\newblock {\em Duke Mathematical Journal}, 81(2):269--322, 1996.

\bibitem{RHP}
D.~Schumayer and D.~A.~W. Hutchinson.
\newblock {{Physics of the Riemann Hypothesis}}.
\newblock {\em Rev. Mod. Phys.}, 83:307--330, 2011.

\bibitem{Shan}
O.~Shanker.
\newblock Good-to-bad gram point ratio for riemann zeta function.
\newblock {\em Experimental Mathematics}, 05 2017.

\bibitem{SP2}
R.~Spira.
\newblock {{Zeros of Sections of the Zeta Function {I}}}.
\newblock {\em Mathematics of Computations}, 18:542--550, 1966.

\bibitem{SP1}
R.~Spira.
\newblock {{Zeros of Approximate Functional Equations}}.
\newblock {\em Mathematics of Computations}, 21(97):41--48, 1967.

\bibitem{SM}
E.~Süli and D.~Mayers.
\newblock {\em {{An Introduction to Numerical Analysis}}}.
\newblock Cambridge University Press, 2003.

\bibitem{Tru1}
T.~Tim.
\newblock On the success and failure of gram's law and the rosser rule.
\newblock {\em Acta Arithmetica}, 148, 02 2011.

\bibitem{Tru2}
T.~Timothy.
\newblock Gram's law fails a positive proportion of the time, 2008.

\bibitem{Titchmarsh1935}
E.~C. Titchmarsh.
\newblock {{The Zeros of the Riemann Zeta-Function}}.
\newblock {\em Proceedings of the Royal Society of London. Series A,
  Mathematical and Physical Sciences}, 151(873):234--255, 1935.

\bibitem{Turing1953}
A.~M. Turing.
\newblock {{Some Calculations of the Riemann Zeta-Function}}.
\newblock {\em Proceedings of the London Mathematical Society, Third Series},
  3.

\end{thebibliography}

\end{document}